\documentclass[11pt,reqno]{amsart}
\usepackage[text={155mm,225mm},centering]{geometry}
\usepackage[mathscr]{eucal}
\usepackage[all]{xy}
\usepackage{mathrsfs}
\usepackage{amsfonts}
\usepackage{amsmath}
\usepackage{amsthm}
\usepackage{amssymb}
\usepackage{latexsym}
\usepackage{textcomp}
\usepackage{stmaryrd}
\usepackage{hyperref}
\usepackage{color}
\usepackage{graphicx}
\usepackage{float} 
\usepackage{subfigure}
\usepackage{pdfpages}
\usepackage{tikz}
\usetikzlibrary{arrows.meta, positioning,calc}
\newtheorem{theorem}{Theorem}[section]
\newtheorem{lemma}[theorem]{Lemma}

\newtheorem{corollary}[theorem]{Corollary}
\newtheorem{proposition}[theorem]{Proposition}
\newtheorem{proposition-definition}[theorem]{Proposition-Definition}
\newtheorem{example-proposition}[theorem]{Example-Proposition}

\newtheorem{question}[theorem]{Question}

\theoremstyle{definition}
\newtheorem{example}[theorem]{Example}
\newtheorem{definition}[theorem]{Definition}

\newtheorem{remark}[theorem]{Remark}

\newtheorem*{ack}{Acknowledgments}

\newcommand{\bbZ}{\mathbb{Z}}
\newcommand{\gl}{\mathfrak{gl}}

\newcommand{\GL}{\mathrm{GL}}

\newcommand{\tos}{\mapsto}

\newcommand{\rmG}{\mathrm{G}}

\newcommand{\ka}{\text{K\"{a}hler}}

\newcommand{\bbK}{\mathbb{K}}

\newcommand{\Hom}{\mathrm{Hom}}

\newcommand{\Alg}{ {\rm Alg}}

\newcommand{\Tr}{\mathrm{Tr} \,}
\newcommand{\bd}{\mathbf{d}}

\newcommand{\QQ}{\mathbb{K} \overline{Q}}
\newcommand{\DQ}{\overline{Q}}
\newcommand{\ldb}{\mathopen{\{\!\!\{}}
\newcommand{\rdb}{\mathclose{\}\!\!\}}}
\newcommand{\bw}{\mathbf{w}}

\newcommand{\mH}{\mathcal{H}}

\newcommand{\Spec}{\mathrm{Spec}\,}

\newcommand{\Mod}{\mathrm{Mod}}

\newcommand{\fkg}{\mathfrak{g}}

\newcommand{\br}{\mathbf{r}}

\newcommand{\bbN}{\mathbb{N}}

\newcommand{\id}{\mathrm{id}}
\newcommand{\rmH}{\mathrm{H}}

\makeatletter

\newcommand{\Rmnum}[1]{\expandafter\@slowromancap\romannumeral #1@}

\newcommand{\cc}{\circ}
\newcommand{\bm}{\mathbf{m}}

\newcommand{\DGA}{\mathrm{DGA}}
\newcommand{\dga}{\DGA^{\leq 0}}
\newcommand{\mC}{\mathcal{C}}
\newcommand{\dgEnd}{\underline{\rm End}}
\newcommand{\pa}{\partial}
\newcommand{\db}{\ldb-,-\rdb}

\newcommand{\bbL}{\mathbb{L}}

\newcommand{\bfL}{\mathbf{L}}
\newcommand{\Lot}{\stackrel{\bfL}{\otimes}}

\newcommand{\KA}{\Omega^1}

\newcommand{\dgHom}{\underline{\mathrm{Hom}}}

\newcommand{\CDGA}{\mathrm{CDGA}}
\newcommand{\cdga}{\CDGA^{\leq 0}}
\newcommand{\sSet}{\mathrm{sSet}}

\newcommand{\bbT}{\mathbb{T}}

\newcommand{\Sym}{\mathrm{Sym}}

\newcommand{\Pol}{\mathrm{Pol}}
\newcommand{\MC}{\mathrm{MC}}

\newcommand{\Com}{\mathrm{Com}}
\newcommand{\bfR}{\mathbf{R}}

\newcommand{\s}{\mathfrak{s}}

\newcommand{\DRep}{\mathrm{DRep}}

\newcommand{\Ch}{\mathrm{Ch}}

\newcommand{\dgDer}{\underline{\mathrm{Der}}}

\newcommand{\dgDDer}{\underline{\mathbb{D}\mathrm{er}}}
\newcommand{\si}{\s ^{-1}}
\newcommand{\cone}{\mathrm{cone}}

\newcommand{\Ts}[1]{\mathrm{T}^{\ast, [#1]}}
\newcommand{\TTs}[1]{\widehat{\mathrm{T}}^{\ast, [#1]}}

\newcommand{\one}{(-1)}

\newcommand{\pmone}{(\pm 1)}

\newcommand{\Aer}{(A^e)^{\otimes r}}

\newcommand{\sfB}{\mathsf{B}}
\newcommand{\HC}{\mathsf{HC}}
\newcommand{\Set}{\mathrm{Set}}
\newcommand{\drep}{\mathpzc{DRep}}

\newcommand{\bfB}{\mathbf{B}}
\newcommand{\bcs}{\backslash}

\newcommand{\dpois}{\mathsf{DPois}}
\newcommand{\pois}{\mathsf{Pois}}

\newcommand{\cof}{\stackrel{\simeq}{\twoheadrightarrow}}
\newcommand{\DGLA}{\mathrm{DGLA}}

\newcommand{\rHC}{\overline{\HC}}

\newcommand{\D}{\mathcal{D}}

\newcommand{\crS}{\mathscr{S}}
\newcommand{\ab}{\mathrm{ab}}

\newcommand{\adga}{\mathrm{ADGA}^{\leq 0}}
\newcommand{\bfOm}{\mathbf{\Omega}}

\newcommand{\tbm}{\widetilde{\bm}}

\newcommand{\Ho}{\mathrm{Ho}}
\newcommand{\CC}{\mathrm{CC}}
\newcommand{\e}{\epsilon}
\newcommand{\CH}{\mathrm{CH}}
\newcommand{\coker}{\mathrm{coker}}

\DeclareMathAlphabet{\mathpzc}{OT1}{pzc}{m}{it}

\allowdisplaybreaks

\title{SHIFTED DOUBLE POISSON STRUCTURES AND NONCOMMUTATIVE POISSON EXTENSIONS}

\author{Leilei Liu}
\address{School of Science, Zhejiang University of Science and Technology, Hangzhou, Zhejiang Province, 310023 P. R. China}
\email{liuleilei@zust.edu.cn}
\author{Jieheng Zeng}
\address{School of Mathematics and Statistics, Hunan Normal University, Changsha, 410000, P.R. China}
\email{zengjh662@163.com}
\author{Hu Zhao}
\address{School of Mathematics and Statistics, Hainan University, Haikou, 570228, P. R. China}
\email{zhaohumath@163.com}

\date{\today}

\begin{document}

\begin{abstract}
	We develop a theory of noncommutative Poisson extensions. For an augmented dg algebra \(A\), we show that any shifted double Poisson bracket on \(A\) induces a graded Lie algebra structure on the reduced cyclic homology.
	Under the Kontsevich--Rosenberg principle, we further prove that the noncommutative Poisson extension is compatible with noncommutative Hamiltonian reduction. 
	Moreover, we show that shifted double Poisson structures are independent of the choice of cofibrant resolutions and that they induce shifted Poisson structures on the derived moduli stack of representations.
\end{abstract}

\maketitle

\setcounter{tocdepth}{2}\tableofcontents


\section{Introduction}\label{sect:intro}

Recently, motivated by three-dimensional $\mathcal{N}=4$ mirror symmetry and geometric representation theory, Poisson geometry on quiver varieties has attracted growing interest across several areas of mathematics. 
At the same time, quiver varieties play a fundamental role in noncommutative geometry.
Following the Kontsevich--Rosenberg principle—which states that a noncommutative geometric structure is meaningful only when it induces the corresponding classical structure on representation schemes—it is natural to seek an intrinsic, noncommutative description of Poisson geometry for quiver varieties.

Let $Q$ be a finite quiver. In \cite{CBEG2007,Van2008Double}, Crawley-Boevey, Etingof, Ginzburg, and Van den Bergh developed noncommutative symplectic and Poisson geometry and showed that preprojective algebras arise from the path algebra $\QQ$ via noncommutative Hamiltonian reduction.
As an immediate consequence, the commutator quotient $\Pi Q/[ \Pi Q,\Pi Q]$ carries a Lie bracket, and the image of the trace map generates the coordinate ring of the quiver variety as a Poisson algebra.
Moreover, this Lie bracket descends to the zeroth reduced cyclic homology $\rHC_0(\Pi Q)$. This motivates the following questions:
\begin{question}\label{que: preproj}
	\begin{enumerate}
		\item Does there exist a Lie algebra structure on the full reduced cyclic homology $\rHC_\bullet(\Pi Q)$?
		\item How can such a Lie algebra structure be explained at the cochain level?
	\end{enumerate}
\end{question}
For non-Dynkin quivers, preprojective algebras are Koszul Calabi--Yau, and hence the works \cite{CEEY2017Der, CheEsh2020} provide affirmative answers to the above questions.
However, for Dynkin quivers, preprojective algebras are not Koszul Calabi--Yau; this work presents an approach that uniformly handles both cases.

Our key tool is the Feigin--Tsygan theorem as developed by Berest--Khachatryan--Ramadoss \cite{BKR2013Der}. For an algebra $A$, the reduced cyclic homology $\rHC_\bullet(A)$ can be computed as the cohomology of
\(\displaystyle{ \frac{\tilde{A}}{S+[\tilde{A},\tilde{A}]}} \),
where \(\tilde{A}\) is a cofibrant resolution of \(A\). For a finite quiver \(Q\) the preprojective algebra \(\Pi Q\) is augmented, and the cobar–bar construction \(\bfOm\bfB\Pi Q\) provides a natural cofibrant resolution of \(\Pi Q\). One of our main results answers Question \ref{que: preproj} positively.

\begin{theorem}[Theorem \ref{thm: preproj}]
	Let $Q$ be a finite quiver.
	There exists a dg Loday algebra structure on $\bfOm\bfB\Pi Q$.
	Furthermore, this dg Loday structure induces a graded Lie algebra structure on $\rHC_\bullet(\Pi Q)$.
\end{theorem}

A crucial observation in the proof is the following isomorphism of coaugmented dg coalgebras for any augmented dg algebra $A$:
\begin{equation}\label{for: BA ext}
	\bfB A \cong S \oplus \s \bar A \langle t \rangle.
\end{equation}

First, we introduce the \(\s\)-construction. Using (\ref{for: BA ext}), the cobar–bar construction is identified with the dg algebra
\((T_S\,\si(\s\bar A\langle t\rangle),\pa+b+\delta,\otimes)\).
When \(A\) is a dg double Poisson algebra, the \(\s\)-construction endows \(\s^dA = A[d]\) with a canonical double Poisson bracket; see Section \ref{sec: s-bracket} for details.

Second, we develop a theory of noncommutative Poisson extensions. Given a dg algebra \(A\), form the free product \(A\langle t\rangle=A\ast_S S[t]\).
We analyze when double Poisson brackets on \(A\) extend to \(A\langle t\rangle\), when NC Poisson structures (see Definition \ref{def: nc pois}) extend, and how these two extension problems are related.
In particular, we show that NC Poisson extensions are compatible with noncommutative Hamiltonian reduction and align with the Kontsevich--Rosenberg principle.

\begin{theorem}[Theorem \ref{thm: com cubic}]
	Let \((A\in\Alg_S,\ldb-,-\rdb,\bw)\) be a noncommutative Hamiltonian algebra.
	Let \(\Theta\in\dgDDer_S(A)\) be a dg double Poisson derivation of degree zero such that \(\Theta (\bw)\in A\otimes A\bw A + A\bw A\otimes A\).
	Then, for any \(\bd\in\bbN^I\), there exists a commutative cube of dg Lie algebras and dg Lie algebra morphisms:
	\[
	\begin{split}
		\xymatrixrowsep{0.8pc}
		\xymatrixcolsep{1.2pc}
		\xymatrix{
			A_\natural\ar[rd]|-{pr} \ar[rr]^-{\rm inclusion}  \ar[dd]|-{\Tr}&& 
			A\langle t; \theta \rangle_\natural  \ar[rd]|-{pr} \ar[dd]|-(0.3){\Tr}\\
			& (A_\bw)_\natural   \ar[dd]|-(0.7){\Tr} \ar[rr]^(0.25){\rm inclusion} 
			&& 
			A_\bw\langle t; \theta \rangle_\natural \ar[dd]|-{\Tr} \\
			A_\bd  ^{\GL_\bd (\bbK)} \ar[rd]|-{pr} \ar[rr]^(0.6){\rm inclusion} && 
			(A\langle t; \theta \rangle)_\bd ^{\GL_\bd (\bbK)}  \ar[rd]|-{pr} \\
			& (A_\bw)_\bd ^{\GL_\bd (\bbK)}    \ar[rr]^-{\rm inclusion} 
			&& (A_\bw \langle t; \theta \rangle)_\bd ^{\GL_\bd (\bbK)}.
		}
	\end{split}
	\]
\end{theorem}

We also prove a more general extension result: the double Poisson bracket extends to the cobar–bar construction.

\begin{theorem}[Theorem \ref{thm: cob bar has dpoiss}, Corollary \ref{cor: dpois to lie}]
	Let \((A\in\adga_S,\pa,\db)\) be a dg double Poisson algebra of degree \(n\).
	Then the cobar-bar construction \(\bfOm \bfB A\) is a dg double Poisson algebra of degree \(n\).
	Furthermore, the reduced cyclic homology of \(A\) carries a graded Lie algebra structure of degree \(n\).
\end{theorem}

The latter statement may be regarded as a noncommutative analogue of the fact that the cohomology of a smooth Poisson manifold carries a graded Lie algebra structure.

All constructions above are carried out with respect to an explicit cofibrant resolution.
In the derived setting,
one expects that noncommutative structures are independent of the choice of cofibrant resolution.
We therefore establish:

\begin{theorem}[Theorem \ref{thm: dpoiss not reso}]
	Let \(A\in \mathrm{DGA}^{\le0}_S\) and
\(n\in \mathbb Z\). Then the space of \(n\)-shifted double Poisson
structures on \(A\) is independent
of the choice of cofibrant resolution up to weak equivalence.
\end{theorem}

For completeness, Section \ref{sec: K-R} explains how shifted double Poisson structures fit naturally into the Kontsevich--Rosenberg principle.
In particular, we prove that a shifted double Poisson structure induces a shifted Poisson structure on the derived moduli stack of representations.

\begin{ack}
	We thank Xiaojun Chen and Farkhod Eshmatov for many helpful discussions.
	In particular, the third author thanks Professor Yongbin Ruan for his continuous support and valuable advice. 
	This work was supported by the National Natural Science Foundation of China (Grant No. 12401050).
\end{ack}

\subsection{Conventions and notation}
\begin{itemize}
	\item $\bbK$ denotes an algebraically closed field of characteristic zero.
	
	\item Algebras are associative algebras over a semisimple ring \(S = \bigoplus_{i\in I}\bbK e_i\), where the \(e_i\) are mutually orthogonal idempotents.
	The category of such algebras over \(S\) is denoted by \(\Alg_S\).
	
	\item We abbreviate $-\otimes_S -$ as $- \otimes -$.
	
	\item We use cohomological grading: all complexes are cochain complexes supported in non-positive degrees. For a complex \(V\), the homology is related to cohomology by \(\rmH_n(V)=\rmH^{-n}(V^\bullet)\).
	
	\item $\Com_S$ denotes the category of cochain complexes of $S$-modules.
	
	\item $\dga_S$ denotes the category of differential graded $S$-algebras supported in non-positive degrees. The notation \(\adga_S\) denotes the category of augmented dg \(S\)-algebras supported in non-positive degrees.

	\item Let \(A\) be a dg algebra. In this work, by a left (right) \(A\)-module, we mean a left (right) dg module over \(A\). The category of  left \(A\)-modules is denoted by \(\Mod(A)\).
	
	\item $\s$ denotes the shift functor. For a complex $V^\bullet \in \Com_S$, \((\s V)^i = V^{i+1}\). 
	
	\item For any $n \in \bbN$, the symmetric group $\mathcal{S}_n$ acts on $V^{\otimes n}$: for $\sigma\in\mathcal{S}_n$ and $v_1\otimes\cdots\otimes v_n\in V^{\otimes n}$,
	\[
	\sigma(v_1\otimes\cdots\otimes v_n) := v_{\sigma^{-1}(1)}\otimes v_{\sigma^{-1}(2)}\otimes\cdots\otimes v_{\sigma^{-1}(n)}.
	\]
	In particular, $\tau$ denotes the cyclic permutation:
	\[
	\tau(a_1 \otimes a_2 \otimes \cdots \otimes a_n)  = a_n \otimes a_1 \otimes \cdots \otimes a_{n-1}.
	\]
\end{itemize}

\section{Noncommutative Poisson Geometry}\label{sec: nc poisson}

In classical Poisson geometry, Poisson structures on a smooth manifold \(X\) can be described algebraically via a Lie bracket on the algebra \(C^\infty(X)\) satisfying the Leibniz rule, or geometrically via a bivector field \(\pi\in\Gamma(X,\wedge^2 TX)\) such that \(\{\pi,\pi\}=0\), where \(\{-,-\}\) denotes the Schouten--Nijenhuis bracket.
Following this philosophy, this section recalls double Poisson geometry in the sense of Van den Bergh \cite{Van2008Double} and explains that shifted double Poisson structures behave well in the derived setting.

\subsection{Double Poisson brackets}\label{subsec: dpoiss alg}

First, we recall the notion of a double Poisson bracket.

\begin{definition}[Van den Bergh]\label{def: dpois bracket}
	Let \(A \in \dga_S\).
	A \emph{dg double Poisson bracket of degree $n$ on \(A\)} is
	an $S$-bilinear map
	\[
	\ldb -,-\rdb: A \otimes A \to A \otimes A
	\]
	of degree \(-n\) satisfying:
	\begin{enumerate}
		\item \(\ldb y,x\rdb = -\one^{|\s^n x|\cdot|\s^n y|}\,\tau\ldb x, y\rdb\);
		\item For any \(x \in A\),
		\(\mH_x := \ldb x,-\rdb : A \to A \otimes A\)
		is a double derivation of degree \(|x|-n\);
		\item \(\ldb -,-\rdb\) satisfies the {\it double Jacobi identity}:
		\begin{align*}
			\ldb x, y, z \rdb & : = \ldb x,\ldb y,z\rdb\rdb_L + \one^{|\s^n z|(| x|+|y|)}(321)\,\ldb z,\ldb x,y\rdb\rdb_L\\
			& \qquad + \one^{|\s^n x|(|y|+|z|)}(123)\,\ldb y,\ldb z,x\rdb\rdb_L = 0;
		\end{align*}
		\item The bracket is compatible with the differential on \(A\):
		\[
		\partial\bigl(\ldb x,y\rdb\bigr)=\ldb \partial x,y\rdb + \one^{|\s^n x|}\,\ldb x,\partial y\rdb.
		\]
	\end{enumerate}
\end{definition}
Here, we write
\[
\ldb x , y \rdb = \ldb x , y \rdb' \otimes \ldb x , y \rdb'' = \mH_x (y)' \otimes \mH_x (y)'' ;\ 
\ldb x,\ldb y,z\rdb\rdb_L = \ldb x, \ldb y, z \rdb' \rdb \otimes \ldb y, z \rdb''.
\]
A dg algebra \(A\) is called a \emph{dg double Poisson algebra of degree \(n\)} if it is equipped with a dg double Poisson bracket of degree \(n\).
A complex \(V\in\Com_S\) is called a \emph{dg double Lie algebra of degree \(n\)} if it has a bracket \(\db\) that satisfies the axioms in Definition \ref{def: dpois bracket}, except (2).

The following examples play important roles in this work.
\begin{example}\label{eg: quiver double poisson}
	Let $Q$ be a finite quiver.
	There exists a double Poisson bracket on $\QQ$ given by:
	for any arrow $a \in Q$,
	\begin{equation*}
		\ldb a, a^{\ast} \rdb = e_{s(a)} \otimes e_{t(a)}, 
		\qquad
		\ldb a^{\ast} , a \rdb = - e_{t(a)} \otimes e_{s(a)};
	\end{equation*}
	for any $f, g\in \overline{Q}  \text{ with } f \neq g^{*}$,
	\(\ldb f, g \rdb = 0.\)
	Extending it by the Leibniz rule defines a double Poisson bracket on \(\QQ\).
\end{example}
Here \(s(-)\) and \(t(-)\) denote the source and target of an arrow, respectively.

\begin{example}
	Let \(\bbK[t]\) be the polynomial algebra generated by \(t\).
	The degree of \(t\) need not be zero; the differential on \(\bbK[t]\) is  zero.
	There is a {dg double Poisson bracket} on \(\bbK[t]\) determined by
	\[
	\ldb t, t \rdb = t \otimes 1 - 1 \otimes t.
	\]
\end{example}

Next, we recall an important property of the double Poisson bracket.
In \cite{Van2008Double}, Van den Bergh showed that a double Poisson bracket on \(A\) canonically induces a Lie algebra structure on the quotient space \(A_{cyc}=A/[A,A]\).
When \(A\) is a dg double Poisson algebra, one obtains a bracket
\(\{-,-\}: A \otimes A \to A\)
defined as the composition
\[
\xymatrix{
	A\otimes A \ar[r]^-{\ldb -,-\rdb} & A\otimes A \ar[r]^-{\bm} & A.
}
\]
Here, \(\bm\) is the multiplication on \(A\).
\begin{lemma}\label{lem: dpois to dgla}
	Let \((A, \pa, \db)\) be a dg double Poisson algebra of degree $n$.
	Then \(A_{cyc}\) is a dg Lie algebra of degree \(n\) with respect to \((\pa, \{-,-\})\).
\end{lemma}

\begin{proof}
	The anti-symmetry of \(\{-,-\}\) on \(A_{cyc}\) is clear; we only show that the Jacobi identity holds.
	First, we prove that for any \(x,y,z\in A\),
	\begin{equation}\label{for: jacobi id}
		\{x,\ldb y,z\rdb\} - \ldb\{x,y\},z\rdb - \one^{|\s^n x|\cdot|\s^n y|}\,\ldb y,\{x,z\}\rdb = 0.
	\end{equation}
	Direct computation yields:
	\begin{align*}
		\{x, \ldb y, z \rdb\}&= (\bm \otimes \id) \big(\ldb x, \ldb y, z \rdb \rdb_L \big) - (\id \otimes \bm) (132) \ldb x, \ldb z, y \rdb \rdb_L \one^{|\s^n y| |\s^n z| };\\
		\ldb \{x, y\}, z \rdb &= - (\bm \otimes \id ) (123) \big(\ldb z, \ldb x, y\rdb \rdb_L \big)\one^{|\s^n z| |\s^n y| + |\s^n x|}\\
		&\qquad + (\id \otimes \bm) \big( \ldb z, \ldb y, x \rdb \rdb_L\big) \one^{|\s^n x| |\s^n y| + |\s^n z| |\s^n y| + |\s^n x|};\\
		\ldb y, \{x, z\} \rdb &= (\id \otimes \bm) \big(\ldb y, \ldb x, z \rdb \rdb_L\big) - (\bm \otimes \id) (132) (\ldb y, \ldb z, x \rdb \rdb_L) \one^{|\s^n z| |\s^n x|}. 
	\end{align*}
	Then (\ref{for: jacobi id}) holds.
	Applying $\bm$ to (\ref{for: jacobi id}) yields the Jacobi identity for \((A_{cyc},\pa, \{-,-\})\).
\end{proof}

\begin{remark}\label{rk: loday}
	 As shown by Van den Bergh \cite{Van2008Double}, the bracket \(\{-,-\}: A \otimes A \to A\) is generally not a dg Lie bracket.
	In fact, it defines a dg Loday algebra structure on \(A\).
	Recall that a \emph{dg Loday algebra of degree \(n\)} is a complex \(V\) equipped with a bilinear map \(\{-,-\}: V \otimes V \to V\) of degree \(n\) satisfying the Jacobi identity and compatibility with the differential \(\pa\).
\end{remark}

Although the Lie bracket in Lemma \ref{lem: dpois to dgla} is induced by the double Poisson bracket, it has since developed into an independent theory, due to Crawley-Boevey \cite{Cra2011}, Berest, Chen, Eshmatov, and Ramadoss \cite{BCER2012Non}.
This notion provides a higher homological extension of the $\rmH_0$-Poisson structure in Crawley-Boevey's work \cite{Cra2011}.
Let \(A\) be a dg algebra.
The complex \(\dgDer_S(A)\) of  \(S\)-derivations on \(A\) is naturally a dg Lie algebra with respect to the commutator bracket.
Denote by \(\dgDer_S(A)^\natural\) the subcomplex of derivations with image in \(S+[A,A]\subseteq A\).
It is straightforward to verify that \(\dgDer_S(A)^\natural\) is an ideal of \(\dgDer_S(A)\).
Therefore, \(\dgDer_S(A)_{\natural}:=\dgDer_S(A)/\dgDer_S(A)^\natural\) is a dg Lie algebra.
The canonical action of \(\dgDer_S(A)\) on \(A\) induces a Lie algebra action of \(\dgDer_S(A)_\natural\) on the quotient space \(\displaystyle A_\natural:=\frac{A}{S+[A,A]}\).
We write \(\varrho:\dgDer_S(A)_\natural\to\dgEnd_S(A_\natural)\) for the corresponding dg Lie algebra homomorphism.

Following \cite{BCER2012Non},  one has the following notion.
\begin{definition}[Berest-Chen-Eshmatov-Ramadoss]\label{def: nc pois}
	Let $A \in \dga_S$.
	Let $n \in \bbZ$. 
	An {\it NC Poisson structure of degree $n$ on $  A$} is a dg Lie algebra structure of degree $n$ on $A_\natural$ such that the adjoint representation
	$\mathrm{ad}:\,  A_\natural \to \dgEnd_S( A_\natural ) $ factors through \(\varrho\).
\end{definition}
In other words, for each \(a\in A\)  representing the class \(\bar a \in A_{\natural}\),
the map \( \mathrm{ad}_{\bar a} = \{ \bar a,-\}: A_\natural \rightarrow A_\natural\) is induced by a derivation \(H_a: A \rightarrow A\).
Throughout this work, a dg algebra \(A\) equipped with an NC Poisson structure of degree \(n\) is called an \emph{NC Poisson algebra of degree \(n\)}.

\begin{example}
	If \(A\) is a dg commutative algebra, then an NC Poisson structure on \(A\) coincides with a dg Poisson bracket on \(A\).
\end{example} 

\begin{example}
	If \(A\) is an ordinary \(S\)-algebra, then an NC Poisson structure on \(A\) coincides with an \({\rmH}_0\)-Poisson structure in \cite{Cra2011}.
	Furthermore, by Lemma \ref{lem: dpois to dgla}, a dg double Poisson bracket induces an NC Poisson structure.
\end{example}
In the quiver case, one has the necklace Lie bracket.
\begin{example}\label{eg: necklace lie}
	The induced Lie bracket on \(\QQ_\natural\) is determined as follows.
For any $a\in Q$,
$\{a,a^{*}\}=1$
and  $\{a^{*},a\}=-1$;
also,
$\{f, g\}= 0$ 
for any $f, g\in \overline{Q}  \text{ with } f \neq g^{*}$.
 By the Leibniz  rule,
for cyclic paths $a_{1}a_{2}\cdots a_{k},b_{1}b_{2}\cdots b_{l}\in \QQ_\natural$
with $a_{i},b_{j}\in \overline{Q}$,
set
\begin{equation}\label{necklace bracket}
	\begin{split}
		&\{a_{1}a_{2}\cdots a_{k},b_{1}b_{2}\cdots b_{l}\}\\
		&=
		\sum_{1\leqslant i \leqslant k,\, 1\leqslant j \leqslant l}
		\{a_{i},b_{j}\}t(a_{i+1})a_{i+1}a_{i+2}\cdots a_{k}a_{1}
		\cdots a_{i-1} b_{j+1}\cdots b_{l}b_{1}\cdots b_{j-1}.
	\end{split}
\end{equation}
\end{example}

In Poisson geometry, many important Poisson spaces are obtained from Hamiltonian reduction.
In the noncommutative setting, the analogue of a Hamiltonian \(\rmG\)-space 
is a double Poisson algebra \(A\) endowed with a noncommutative moment map.
\begin{definition}[Crawley--Boevey--Etingof--Ginzburg, Van den Bergh]\label{def: nc moment map}
	Let \((A, \pa, \db)\) be a dg double Poisson algebra of degree \(n\). A \emph{noncommutative moment map} is an element
	\(\bw=\sum_{i\in I}\bw_i\in\bigoplus_{i\in I} e_i A e_i\)
	such that
	\(\ldb\bw_i,p\rdb = p\otimes e_i - e_i\otimes p\) for every \(p\in A\).
\end{definition}
In this case, \(A\) is called a \emph{noncommutative Hamiltonian algebra of degree \(n\)}.
Following the Kontsevich--Rosenberg principle, noncommutative reduction theory is developed in \cite{CBEG2007, Van2008Double, Van2008Quasi}.
\begin{definition}[Crawley-Boevey-Etingof-Ginzburg, Van den Bergh]
	Let $(A, \pa, \ldb-,-\rdb, \bw)$ be a noncommutative Hamiltonian algebra of degree $n$.
	The {\it noncommutative Hamiltonian reduction of $A$} is the quotient algebra
	\(\displaystyle{\frac{A}{A \bw A}}\), denoted by $A_\bw$.
\end{definition}
This notion is justified by the following proposition.
\begin{proposition}\label{prop: CBEG}
	Let $(A, \pa, \ldb-,-\rdb, \bw)$ be a noncommutative Hamiltonian algebra. Let $A_\bw$ be the noncommutative Hamiltonian reduction. Then there exists an NC Poisson structure on $A_\bw$; furthermore, the canonical projection $A \to A_\bw$ descends to a dg Lie algebra morphism
	\begin{equation}\label{for: nc red for lie}
		pr: A_\natural \to (A_\bw)_\natural.
	\end{equation}
\end{proposition}
\begin{proof}
	The proof of \cite[Theorem 7.2.3]{CBEG2007} carries over to the dg setting.
\end{proof}

\begin{example}\label{eg: A_3}
	According to \cite{CBEG2007, Van2008Double},
	for the double Poisson algebra $\QQ$ in Example \ref{eg: quiver double poisson},
	\(\bw=\sum_{a\in Q}(a a^\ast - a^\ast a)\)
	is a noncommutative moment map.
	Consider the quiver \(Q\) in Figure \ref{A3}.
	Doubling \(Q\) yields \(\overline{Q}\) as shown in Figure \ref{DA3}.
	\begin{figure}[H]
		\centering
		\subfigure[$Q$]{
			\begin{tikzpicture}[
				v/.style={draw, circle, inner sep=1pt, fill=white, font=\small},
				arr/.style={-Stealth, thick}
				]
				\node[v] (inf) at (0,2) {$e_{\infty}$};
				\node[v] (0)   at (0,0) {$e_{0}$};
				\node[v] (1)   at (-1, -2) {$e_{1}$};
				\node[v] (2)   at (1, -2) {$e_{2}$};
				
				\draw[arr] (inf) -- node[midway, left] {$p$} (0);  
				\draw[arr] (0) -- node[midway, above, yshift=-7pt, xshift=-10] {$a_0$} (1);  
				\draw[arr] (1) -- node[midway, below] {$a_1$} (2);  
				\draw[arr] (2) -- node[midway, right]{$a_2$} (0);  
			\end{tikzpicture}
			\label{A3}}
		\qquad\qquad\qquad
		\subfigure[$\overline Q$]{
			\begin{tikzpicture}[
				v/.style={draw, circle, inner sep=1pt, fill=white, font=\small},
				arr/.style={-Stealth, thick}
				]
				\node[v] (inf) at (0,2) {$e_{\infty}$};
				\node[v] (0)   at (0,0) {$e_{0}$};
				\node[v] (1)   at (-1, -2) {$e_{1}$};
				\node[v] (2)   at (1, -2) {$e_{2}$};
				
				\draw[arr, bend left =20]  (inf) to node[midway, right] {$p$} (0);  
				\draw[arr, bend left = 20] (0) to node[midway, above, xshift = -8, yshift=-7] {$a_0$} (1);
				\draw[arr, bend left = 20] (1) to node[midway, below, yshift= 1.5] {$a_1$} (2);
				\draw[arr, bend left = 20] (2) to node[midway, right] {$a_2$} (0);
				
				\draw[arr, bend left=20] (0) to node[midway, left] {$p^*$} (inf);
				\draw[arr, bend left=20] (1) to node[midway, above, xshift = -9, yshift= -10] {$a_0^*$} (0);
				\draw[arr, bend left=20] (2) to node[midway, below, yshift=1] {$a_1^*$} (1);
				\draw[arr, bend left=20] (0) to node[midway, right] {$a_2^*$} (2);
				
			\end{tikzpicture}
			\label{DA3}}
		\caption{Quiver $Q$ and the doubled quiver $\overline{Q}$.}
	\end{figure}
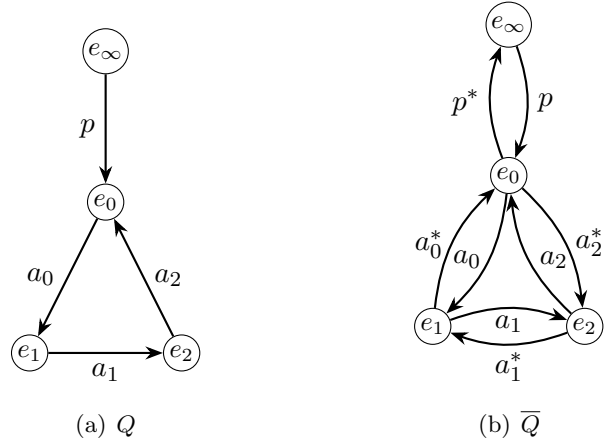
	Then the noncommutative moment map is
	\begin{equation*}\label{w for A_3}
		\begin{split}
			\bw & = \sum_{i=0}^{2}[a_i, a_i^\ast] + p p^\ast - p^\ast p\\
			& =  \big( - a^\ast _0 a_0 +  a_2 a_2 ^{\ast}  + p p^\ast \big)
			+ \big(a_0 a^\ast _0 - a_1^\ast a_1 \big)
			+ \big( a_1a_1^\ast -a_2 ^\ast a_2\big)
			+ \big(  - p^\ast p\big)\\
			& = \bw_0 + \bw_1 + \bw_2 + \bw_\infty.
		\end{split}
	\end{equation*}
\end{example}

\subsection{Higher noncommutative Poisson structures}\label{subsec: higher dpoiss}

Double Poisson brackets appear more fundamental than NC Poisson structures. In the derived setting, a natural question is whether double Poisson brackets are stable under quasi-isomorphisms. The key observation is that, under smoothness assumptions, double Poisson brackets can be constructed by solving a Maurer–Cartan equation.

We first recall the construction of noncommutative cotangent bundles in \cite{CBEG2007, Van2008Double} and the correspondence between double Poisson brackets and elements in these bundles. In algebraic geometry, polyvector fields on a smooth variety are identified with functions on the cotangent bundle. Following this intuition, a noncommutative cotangent bundle is defined via noncommutative vector fields. According to Crawley-Boevey-Etingof-Ginzburg and Van den Bergh, noncommutative vector fields on an algebra are double derivations.

Let \(A \in \dga_S\). The \emph{noncommutative differential} \(\KA(A)\) of \(A\) is defined as the kernel of the multiplication map \(\bm: A \otimes A \to A\) in the category \(\Mod(A^e)\) of left $A^e$-modules.
\(A\) is \emph{smooth} if \(\KA(A)\) is projective and finitely presented as a left $A^e$-module. The complex \(\dgDDer_S(A) := \dgDer_S(A, A \otimes A)\) of double derivations is isomorphic to the bimodule dual of \(\KA(A)\), namely \(\dgDDer_S(A) \cong \KA(A)^{\vee} = \dgHom_{A^e} (\KA(A), A \otimes A)\) as right $A^e$-modules. The right $A^e$-module structure on \(\dgDDer_S\) is induced from the inner $A$-bimodule structure on \(A \otimes A\).

The noncommutative cotangent bundle is defined as follows.
\begin{definition}[Crawley-Boevey-Etingof-Ginzburg, Van den Bergh]
	Let \(A \in \dga_S\) and \(n \in \bbZ\). The \emph{noncommutative \(n\)-shifted cotangent bundle of \(  A\)} is defined by
	\[
	\Ts{n}_{  A} : = T^\bullet_A\Bigl(\dgDDer_S(A)[-n]\Bigr).
	\]
\end{definition}
Note that \(\Ts{n}_{  A}\) is a dg algebra with respect to the tensor product, and it carries a natural weight grading:
\begin{equation}\label{for: weight grading}
	\Ts{n}_{  A} = \bigoplus_{p\geq 0} \Bigl(\Ts{n}_{  A}\Bigr)^{(p)},\ 
\text{where}\ \Bigl(\Ts{n}_{  A}\Bigr)^{(p)} = \Bigl(\dgDDer_S(A)[-n]\Bigr)^{\otimes_A p}.
\end{equation}
It is straightforward to verify that the weight grading (\ref{for: weight grading}) descends to the quotient space \(\Ts{n}_{  A,\,cyc}: = \Ts{n}_{  A}/[\Ts{n}_{  A}, \Ts{n}_{  A}]\).
Denote the completion by
\(\TTs{n} _{  A} : = \prod_{i \geq 0} \Big(\dgDDer_S(A)[-n]\Big)^{\otimes_A i}.\)
There is a natural weight filtration on \(\TTs{n}_{  A}\): for any \(p \in \bbZ_{\geq 0}\),
\begin{equation}\label{for: wei fil on nc cot}
	F^p\TTs{n}_{  A} : = \prod_{i\geq 0} \Bigl(\dgDDer_S(A)[-n]\Bigr)^{\otimes_A (p+i)}.
\end{equation}
In some literature, for example \cite{Pri2020Shi}, the (completed) noncommutative cotangent bundle is called shifted noncommutative multiderivations.

In differential geometry, there is a canonical Poisson structure on a cotangent bundle over a smooth manifold, and this Poisson structure plays a central role in Poisson geometry. 
Parallel to the noncommutative symplectic approach in \cite{CBEG2007}, Van den Bergh studied the double Poisson geometry on the noncommutative cotangent bundle in \cite{Van2008Double,Van2008Quasi}.
A crucial property is as follows.
\begin{proposition}\label{prop: double schouten}\cite[Proposition 3.2.2]{Van2008Double}
	Let $A \in \dga_S$. Then for any integer $n$, the $n$-shifted noncommutative cotangent bundle $\Ts{n}_{  A}$ is a dg double Poisson algebra of degree $n$.
\end{proposition}

In Van den Bergh's work \cite{Van2008Double}, this bracket is called the {\it double Schouten--Nijenhuis bracket}.
If \(A\) is assumed to be smooth, then
the double Poisson bracket on \(\Ts{n}_{  A}\) can be described explicitly.
Assume that as a left \(A^e\)-module,
\(\KA(  A)\) is generated by \(\{\delta a_i  = a_i \otimes 1 - 1 \otimes a_i\ \big\vert\ i =1, \dots ,r\}\).
Then the right \(A^e\)-module \(\dgDDer_S(A)\) is generated by \(\{\D_{a_i}\}\),
where
\[
\D_{a_i} (\delta a_j) =\delta_{i,j} e_{t(a_i)} \otimes e_{s(a_i)}.
\]
Then, by the Leibniz rule, the double Schouten–Nijenhuis bracket is determined by evaluation on generators \(\{a_1,\dots,a_r,\s^{-n}\D_{a_1},\dots,\s^{-n}\D_{a_r}\}\):
\begin{equation}\label{for: dpoiss on nc cot}
	\ldb a_i, a_j \rdb = 0,\qquad
	\ldb \s^{-n}\D_{a_i}, \s^{-n}\D_{a_j} \rdb = 0,\qquad
	\ldb \s^{-n}\D_{a_i}, a_j \rdb = \delta_{i,j} e_{t(a_i)} \otimes e_{s(a_i)}.
\end{equation}

A key observation in Van den Bergh's work \cite{Van2008Double} is that elements
in noncommutative cotangent bundles correspond to brackets on \(A\).
Related higher and homotopy-invariant versions of this viewpoint appear
in Pridham \cite[Section 2]{Pri2020Shi} and Yeung \cite[Section 2]{Yeu2022Pre}.
\begin{proposition}\label{prop: mc for dpois}
	Let $A \in \dga_S$ be smooth. Then for any positive integer $r$, there is a bijection of sets
	\begin{equation}\label{eq: nc cot to map}
		\Bigl(\Ts{n+1}_{  A,\,cyc} [n+1]\Bigr)^{(r)} \cong \dgHom_{(A^e)^{\otimes r}}\Bigl((\KA(  A)[1+n])^{\otimes r}, A^{\otimes r}\Bigr)^{\bbZ_r} [n+1].
	\end{equation}
	Furthermore, for
	\(\displaystyle{P \in \Bigl(\Ts{n+1}_{  A,\,cyc} [n+1]\Bigr)^{(2)}}\) a cocycle of cohomological degree $1$,
	the associated bracket $\db_P$ by (\ref{for: ass bracket})  is a dg double Poisson bracket if and only if 
	\(\frac{1}{2}\{P, P\} = 0.\)
\end{proposition}

The proof of \cite[Proposition 4.1.1]{Van2008Double} adapts to the dg setting; for completeness, we include a sketch.
Before that, we recall the necessary constructions.
For any left $A^e$-module $M$ and for any positive integer $r$,
$M^{\otimes r +1}$ is canonically a left $\Aer$-module and a left \(A^e\)-module.
In particular,  the left   $\Aer$-module  and the left $A^e$-module structure on $A^{\otimes (r+1)}$ are as follows.
Suppose that $a = a' \otimes a'' \in A^e$, and $ x = x_1 \otimes \cdots \otimes  x_{r+1} \in A^{\otimes (r+1)}$,
the ${\rm i^{th}}$-copy of \(\Aer\) acts on $A^{\otimes r+1}$ by
\[
(a' \otimes a'') \bullet_i (x_1 \otimes \cdots \otimes x_{r+1}) = \one^{|a'| |a''| + |a| |x_{< i+1}|} x_1 \otimes\cdots \otimes x_{i} a'' \otimes a' x_{i+1} \otimes \cdots \otimes x_{r+1};
\]
the canonical left $A^e$-module structure on $A^{\otimes (r+1)}$ is given by
\[
(a' \otimes a'') . (x_1 \otimes \cdots \otimes x_{r+1})  = \one^{|a''| |x|}(a 'x_1 )\otimes \cdots \otimes (x_{r+1} a'').
\] 
Here, \(|x_{< i+1}| = |x_1| + \cdots + |x_i|\).
These two module structures clearly commute.
Furthermore, there exists  a canonical morphism of left  $A^e$–modules
\begin{equation}\label{for: Phi}
	\Phi : M^{\otimes_A r} \longrightarrow \dgHom_{\Aer}\Bigl(\bigl(M^{\vee}\bigr)^{\otimes r}, A^{\otimes (r+1)}\Bigr).
\end{equation}
For any \(\Theta_1 \otimes \cdots \otimes \Theta_r \in M^{\otimes_A r}\), the image \(\Phi(\Theta_1 \otimes \cdots \otimes \Theta_r)\) is given by
\[
\alpha_1 \otimes \cdots \otimes \alpha_r \mapsto \pmone\, \Theta''_1(\alpha_1) \otimes \Theta'_1(\alpha_1)\,\Theta''_2(\alpha_2) \otimes \cdots \otimes \Theta'_{r-1}(\alpha_{r-1})\,\Theta''_r(\alpha_r) \otimes \Theta'_r(\alpha_r).
\]
Here,  \(\pmone\) follows from the Koszul sign rule.

\begin{proof}[Proof for Proposition \ref{prop: mc for dpois}]
	Let $M = \KA(  A)$.
	First, the right-hand side of (\ref{eq: nc cot to map}) is related to brackets on \(A\).
	For any \(r \in \bbZ_{\geq 0}\) and a \(\bbZ_r\)-equivariant map
	\[
	P \in \dgHom_{(A^e)^{\otimes r}}\Bigl((\KA(  A)[1+n])^{\otimes r}, A^{\otimes r}\Bigr)^{\bbZ_r} [n+1],
	\]
	the associated bracket on $A$ is given by
	\begin{equation}\label{for: ass bracket}
	x_1 \otimes \cdots \otimes x_r \longmapsto P \cc (\s^n \otimes \dots \s^n ) \cc (d \otimes \dots \otimes d) (x_1 \otimes \cdots \otimes x_r) = \pmone \Lambda\Bigl(\s^{n} d x_1 \otimes \cdots \otimes \s^{n}  dx_r\Bigr).
	\end{equation}
	Here, \(d: A \to \KA(  A) [1]\) is the de Rham differential; \(\bbZ_r\) acts on \(\KA(  A)[1+n])^{\otimes r}\text{ and } A^{\otimes r}\) by cyclic permutation.

	Since \(A\) is smooth,
	$\Phi$ of (\ref{for: Phi}) is an isomorphism.
	On the other hand, one has the identification
	\(A \otimes_{A^e} A^{\otimes (r+1)} \cong A^{\otimes r}\).
	Let
	\[
	\Psi: A \otimes_{A^e} \Bigl(\dgDDer_S(A)[-1-n]\Bigr)^{\otimes_A r} \longrightarrow \dgHom_{\Aer}\Bigl(\KA(  A)[1+n]^{\otimes r}, A^{\otimes r}\Bigr)
	\]
	be the natural composition.
	Note that
	the weight decomposition of $\Ts{n+1}_{  A}$ descends well to \(\Bigl(\Ts{n+1}_{  A,\,cyc} [n+1]\Bigr)\);
	\(\Bigl(\Ts{n+1}_{  A,\,cyc} [n+1]\Bigr)^{(r)}\) is the weight $r$ component of $\Ts{n+1}_{  A,\,cyc} [n+1]$, by definition,
	\[
	\Bigl(\Ts{n+1}_{  A,\,cyc} [n+1]\Bigr)^{(r)} = \dgHom_{({A}^e)^{r}} (\KA(A/S)[1+n]^{\otimes r}, A^{\otimes r})_{\bbZ_r} [n+1].
	\]
	Then averaging  the $\bbZ_r$-action,
	one has the morphism
	\[
	\Upsilon: \Bigl(\Ts{n+1}_{  A,\,cyc} [n+1]\Bigr)^{(r)} \longrightarrow \dgHom_{(A^e)^r}\Bigl(\KA(  A)[1+n]^{\otimes r}, A^{\otimes r}\Bigr)^{\bbZ_r} [n+1].
	\]
	By our assumptions, $\Upsilon$ is an isomorphism.
	The second assertion is proved exactly as in \cite[Theorem 4.2.3]{Van2008Double}.
\end{proof}

\begin{example}\label{eg: dpoiss field on QQ}
	Given a finite quiver $Q$, the noncommutative bivector field for the double Poisson bracket in Example \ref{eg: quiver double poisson} is
	\[
	P_{\DQ} = \sum_{a \in Q} - \si\D_{a} \otimes \si \D_{a^\ast}.
	\]
\end{example}

Next, we recall double Poisson structures in the derived setting.
Note that Proposition \ref{prop: mc for dpois} shows that a dg double Poisson bracket is governed by a Maurer--Cartan equation, a viewpoint that naturally extends to the derived setting (\cite{BCER2012Non,CEEY2017Der, CheEsh2020, Pri2020Shi,Yeu2022Pre}).
However, noncommutative \(\ka\) differentials and double derivations in previous constructions are replaced by noncommutative cotangent complexes and noncommutative tangent complexes.
Following Andr\'e--Quillen (co)homology theory,
for a dg algebra \(A \in \dga_S\), the {\it (absolute) cotangent complex  \(\bbL_{S \bcs A}\) of  \(A\)}  as an object in \(\Ho(\Mod(A^e))\), is defined to be the homotopy fiber of the multiplication \(\bm: A \otimes A \to A\):
\begin{equation}
	\bbL_{S \bcs A}\longrightarrow  A \otimes A \stackrel{\bm}{\longrightarrow} A \stackrel{+1}{\longrightarrow}
\end{equation}
Then, the {\it noncommutative tangent complex \(\bbT_{S \bcs A}\) of \(A\) }is defined to be the bimodule dual:
\[
\bfR \dgHom_{A^e} (\bbL_{S \bcs A}, A\otimes A).
\]
Consequently,
for an arbitrary cofibrant resolution $Q \cof A$ in $\dga_S$,
the noncommutative cotangent complex is quasi-isomorphic to
\(A^e \otimes_{Q^e} \KA(Q)\);
the noncommutative tangent complex is quasi-isomorphic to \(\dgDDer_S (Q) \otimes_{Q^e} A^e\).
Note that in the derived setting, 
the \(n\)-shifted noncommutative cotangent bundle \(\Ts{n}_{A}\) of \(A\) as an object in \(\Ho(\DGA_S)\)
\footnote{\( \DGA_S\) is the category of dg algebras over \(S\).}
is constructed by choosing an explicit resolution of the noncommutative tangent complex.
A standard fact is that \(\Ts{n}_{A}\) 
is quasi-isomorphic to \(T_Q (\dgDDer_S (Q)[-n])\).

Now, we briefly recall the Maurer-Cartan formalism.
See \cite{Saf2017Lec} for further details.
Let \(\Omega(\Delta_n)\) be the commutative dg algebra of polynomial differential forms on the standard \(n\)–simplex \(\Delta_n\).
For a nilpotent dg Lie algebra \((L, \pa, \{-,-\})\),
set \(\MC(L) = {\{x \in L\ \big\vert\ \pa x + \frac{1}{2} \{x,x\} = 0\}}.\)
The \emph{Maurer–Cartan space of \(L\)} is the simplicial set
\(\underline{\MC}_\bullet(L)=\MC\big(L\otimes_\bbK\Omega(\Delta_\bullet)\big).\)
In general, the dg Lie algebra of interest is not nilpotent but pronilpotent.
The latter means \(L\) admits a cofiltration \(L_0\leftarrow L_1\leftarrow\cdots,\) such that \(L=\lim_n L_n\) and each \(L_n\) is nilpotent. Then for a pronilpotent \(L\), the \emph{Maurer–Cartan space of \(L\)} is \(\underline{\MC}_\bullet(L):=\lim_n\underline{\MC}_\bullet (L_n) \), the limit of \(\{\underline{\MC}_\bullet(L_n)\}\).
A standard fact is that given a graded dg Lie algebra \(L\), its completion \(L^{\geq 2}\) in weights \(\ge 2\) is pronilpotent.
The cofiltration is given by
\(L^{\geq 2}/ L^{\geq 3} \leftarrow L^{\geq 2}/L^{\geq 4} \leftarrow \cdots.\)

Now apply the above Maurer--Cartan formalism to the noncommutative shifted cotangent bundles.
The following notion serves as the moduli of shifted double Poisson structures.
\begin{definition}\label{def: shifted dpois}
	Let $A \in \dga_S$ and \(Q\) a cofibrant resolution of \(A\).
	Let $n$ be an integer.
	The \emph{space of $n$–shifted double Poisson structures on $ A$}
	is the simplicial set
	\[
	\dpois(A; n) : = \underline{\MC}_\bullet \Bigl(F^2\TTs{n+1}_{ Q,\,cyc}[n+1]\Bigr).
	\]
\end{definition}
At first sight, \(\dpois\) seems to depend on the choice of cofibrant resolution. However, we have the following theorem.
\begin{theorem}\label{thm: dpoiss not reso}
	Let $A \in \dga_S$.
	Let $n$ be an integer.
	For different cofibrant resolutions \(Q_1\cof A,\ Q_2 \cof A\) in \(\dga_S\), \( \underline{\MC}_\bullet \Bigl(F^2\TTs{n+1}_{ Q_1 ,\,cyc}[n+1]\Bigr)\) is weakly equivalent to \( \underline{\MC}_\bullet \Bigl(F^2\TTs{n+1}_{ Q_2 ,\,cyc}[n+1]\Bigr)\) in the category \(\sSet\) of simplicial sets.
\end{theorem}
\begin{proof}
	For simplicity, we prove the theorem for \(n=0\).
	Since \(\dga_S\) is a fibrant model category,
	according to \cite[Lemma 2.24]{DS1995Hom}, there exist quasi-isomorphisms \(F: Q_1 \to Q_2\) and \(G: Q_2 \to Q_1\) such that
	\(F\cc G\) is homotopy equivalent to \(\id : Q_2 \to Q_2\) and \(G \cc F\) is homotopy equivalent to \(\id : Q_1 \to Q_1\).
	
	First, there exists a canonical morphism of dg algebras from \(T_{Q_1}\dgDDer_S(Q_1)\) to \(T_{Q_2}\dgDDer_S(Q_2)\).
	Using \(G\), there is a canonical morphism in \(\Mod(Q_1^e)\):
	\[
	Q_1^e\otimes_{Q_2^e}\KA(  Q_2)\to\KA(  Q_1),\qquad (a'\otimes a'')\otimes\delta b\mapsto (a'\otimes a'')\otimes\delta G(b).
	\]
	Therefore, there is a composition of \(S\)-linear maps:
	\begin{align*}
		\widetilde{G}: &\ \dgDDer_S(Q_1) \cong \dgHom_{Q_1 ^e} \big( \KA(  Q_1), Q_1 ^{\otimes 2}\big) \longrightarrow  \dgHom_{Q_1 ^e} \big(Q_1 ^e \otimes_{Q_2 ^e}\KA(  Q_2), Q_1 ^{\otimes 2}\big)\\
		& \longrightarrow \dgHom_{Q_2 ^e} \big( \KA(  Q_2), Q_1 ^{\otimes 2} \big) \longrightarrow \dgHom_{Q_1 ^e} \big( \KA(  Q_2), Q_2 ^{\otimes 2} \big) \cong \dgDDer_S(Q_2).
	\end{align*}
	Similarly, one has \(\widetilde{F}:\dgDDer_S(Q_2)\to\dgDDer_S(Q_1)\).
	The above construction yields dg algebra morphisms; by an abuse of notation,
	\[
	\widetilde{G} : T_{Q_1}\dgDDer_S(Q_1) \to T_{Q_2} \dgDDer_S(Q_2),\qquad \widetilde{F}: T_{Q_2}\dgDDer_S(Q_2) \to T_{Q_1} \dgDDer_S(Q_1).
	\] 
	It is clear that \(\widetilde{F}\) and \(\widetilde{G}\) descend to the completions and the cyclic quotients.
	By construction, \(\widetilde{F}\) and \(\widetilde{G}\) are homotopy inverses of each other.
	
	Second, if \(\widetilde{F}\) and \(\widetilde{G}\) preserve the double Poisson brackets, then by the above analysis,
	they induce a homotopy equivalence of dg Lie algebras between
	\(\TTs{0}_{  Q_1 ,\,cyc}\) and \(\TTs{0}_{  Q_2,\,cyc}\);
	therefore, \(\underline{\MC}_\bullet\big(\TTs{0}_{  Q_1,\,cyc}\big)\) and \(\underline{\MC}_\bullet\big(\TTs{0}_{  Q_2,\,cyc}\big)\) are weakly equivalent as simplicial sets.
	
	Now, we prove that \(\widetilde{F}\) and \(\widetilde{G}\) preserve double Poisson brackets.
	Due to the Leibniz rule of double Poisson brackets and (\ref{for: dpoiss on nc cot}),
	we only need to verify the commutativity of the following diagram:
	\begin{equation*}
		\xymatrix{
			Q_1 \otimes \dgDDer_S(Q_1)  \ar[rr]^-{ \db} \ar[d]_-{F \otimes \widetilde{G}}&& Q_1 \otimes Q_1 \ar[d]^-{F \otimes F}\\
			Q_2 \otimes \dgDDer_S(Q_2) \ar[rr]^-{ \db}&& Q_2 \otimes Q_2.
		}
	\end{equation*}
	By the \(\dgHom\)–\(\otimes\) duality, it is equivalent to the commutativity of the following diagram:
	\begin{equation*}
		\xymatrix{
			\dgDDer_S(Q_1)  \ar[rr]^-{ \id} \ar[d]_-{\widetilde{G}}&& \dgHom_S (Q_1, Q_1 \otimes Q_1) \ar[d]^-{F \otimes F  \cc (-)\cc G }\\
			\dgDDer_S(Q_2) \ar[rr]^-{ \id}&& \dgHom_S (Q_2, Q_2 \otimes Q_2).
		}
	\end{equation*}
	For a general element \(b \D_a c \in \dgDDer_S(Q_1)\),
	\[
	\widetilde{G}(b \D_a c) = F \otimes F  \cc (b \D_a c) \cc G.
	\]
	Therefore, the above diagram commutes, and the statement follows.
\end{proof}

A canonical problem in the derived setting is whether a geometric structure is stable under weak equivalences.
Note that
for two quasi-isomorphic dg algebras \(A, B \in \dga_S\), their cofibrant resolutions \(Q_A,\ Q_B\) are quasi-isomorphic.
By the same analysis as in Theorem \ref{thm: dpoiss not reso}, we have:
\begin{corollary}
	Let $A, B \in \dga_S$.
	Let $n$ be an integer.
	For any quasi-isomorphism \(F: A \to B\) in \(\dga_S\), \(\dpois(A; n) \) is isomorphic to \(\dpois(B; n)\) in the homotopy category \(\Ho(\sSet)\).
\end{corollary}

\begin{remark}
	It is more natural to regard $\dpois(A; n)$ as an object in the homotopy category $\Ho(\sSet)$ or in the associated $\infty$-category. Throughout this work, we fix an explicit cofibrant resolution $Q \cof A$ and take $\dpois(Q; n) \in \sSet$ as a model for $\dpois(A; n)$.
\end{remark}

Finally, it is natural to seek an analogue of Lemma \ref{lem: dpois to dgla} for shifted double Poisson structures. In \cite{BCER2012Non}, Berest-Chen-Eshmatov-Ramadoss introduced the homotopy version of NC Poisson structures.

\begin{definition}
	Let \(A \in \dga_S\).
	Let \(n\) be an integer.
	An {\it \(n\)-shifted  NC \(P_\infty\)-structure on \(A \in \dga_S\) }is  an $L_{\infty}$-structure $\{l_n: \wedge^d (A_{\natural}[n]) \to (A_{\natural}[n]) \}_{d \geq 1}$ such that
	for all $a_1,..,a_{d-1}\,\in\,A$ homogeneous, $l_d( \s^d\bar{a}_1, \cdots,  s^d\bar{a}_{d-1},\mbox{--}):A_{\natural}[n] \to A_{\natural}[n]$ is induced on
	$A_{\natural}[n]$ by a derivation $H_{a_1, \cdots ,a_{d-1}}:A[n] \to A[n]$.
\end{definition}

The following corollary is clear.
We remark that the \(L_\infty\)-structure on \(A_{\natural}[n]\) is well known,
see \cite[Remark 2.12]{Pri2020Shi}.
\begin{corollary}\label{cor: dpois to L_inf}
	Let $A \in \dga_S$.
	Let $n$ be an integer.
	If \(A\) admits an \(n\)-shifted double Poisson structure, then \(A\) canonically admits an \(n\)-shifted  NC \(P_\infty\)-structure.
\end{corollary}

\section{Kontsevich--Rosenberg Principle}\label{sec: K-R}

The Kontsevich--Rosenberg principle states that a noncommutative geometric structure is meaningful precisely when it induces the classical counterpart on representation schemes.
In this section, we recall the construction of derived representation schemes  in \cite{BKR2013Der} and prove that shifted double Poisson structures on \(A\) induce shifted Poisson structures on derived representation schemes.
In particular, we show that shifted double Poisson structures are related to the shifted Poisson structures on the derived moduli stacks of representations.

\subsection{Derived representation schemes}\label{subsec: drep}

Let $V\in \Com_S$ be a complex of finite total dimension, and let $\dgEnd_S (V)$ denote the complex of  endomorphisms of $V$. For $A \in \dga_S$, consider the following representation functor:
\[
\DRep_V^{  A}: \cdga_\bbK \to \Set,\quad C \mapsto \DRep_V ^{  A}(C)  : = \Hom_{\dga_S}(A, \dgEnd_S (V)\otimes_\bbK C).
\]
As established in \cite[Theorem 2.1]{BKR2013Der}, this functor is representable; that is, there exists a commutative dg algebra \(A_V\in\cdga_\bbK\) such that
\( \DRep_V^{  A}(C)=\Hom_{\cdga_\bbK}(A_V, C). \)
When $V$ is fixed, this yields a functor
\((-)_V: \dga_S \to \cdga_\bbK.\)
A key result in \cite{BKR2013Der} is the following.

\begin{proposition}\label{prop: drep quillen}\cite[Theorem 2.2]{BKR2013Der}
	The functors
	\[
	(-)_V: \dga_S \rightleftarrows \cdga_\bbK : \dgEnd_S (V)\otimes -
	\]
	form a Quillen adjunction. Furthermore, for an ordinary algebra $A \in \Alg_S \subset \dga_S$,
	$\rmH^0 (\bfL (A)_V) \cong A_V.$
\end{proposition}

For future reference, we recall the adjunction explicitly.
\begin{proposition}\cite[Proposition 2.1]{BKR2013Der}\label{prop: alg duality}
	For any \(A \in \dga_S\), \(B \in \dga_S\) and \(R \in \cdga_\bbK\), there are natural bijections
	\begin{enumerate}
		\item[(a)] \(\Hom_{\dga_S}(\sqrt[V]{A}, B) \cong \Hom_{\dga_S}(A, \dgEnd_S(V) \otimes B)\);
		\item[(b)] \(\Hom_{\cdga_\bbK}(A_V, R) \cong \Hom_{\dga_S}(A, \dgEnd_S (V) \otimes R)\).
	\end{enumerate}
\end{proposition}
We recall the explicit constructions of \(\sqrt[V]{-},\ (-)_V\) in Section \ref{subsec: present}.
The derived affine scheme \(\Spec A_V\)
is called the \emph{derived representation scheme of \(A\) in \(V\)}, still denoted \(\DRep_{V}^{  A}\).
It appears that the derived representation scheme depends on the choice of \(V\). However,
Proposition 2.3 in \cite{BKR2013Der} states that if \(V\) and \(W\) are quasi-isomorphic in \(\Com_S\), then \(\DRep_{V}^{  A}\) and \(\DRep_{W}^{  A}\) are equivalent as derived schemes.
Throughout  this work, we  consider only \(V = \bbK^\bd = \bigoplus_{i \in I} \bbK^{d_i}\) for \(\bd \in \bigoplus_{i\in I} \bbZ_{\geq 0}\).
The corresponding commutative dg algebra is denoted by \(A_\bd\).

Since \(\GL_\bd (\bbK) = \prod_i \GL(\bbK^{d_i})\) acts on \(\DRep^{A} _\bd\) by conjugation,
we consider the following functor:
\begin{equation*}
	(-)^\GL _\bd : \dga_S \to \cdga_\bbK,\quad A \tos (A_\bd)^{\GL_\bd (\bbK)}. 
\end{equation*}
Here \((A_\bd)^{\GL_\bd(\bbK)}\) denotes the subalgebra of \(\GL_\bd(\bbK)\)-invariants.
\(\Spec (A)^\GL _\bd\) is called the {\it derived character scheme of $A$ in $\bbK^\bd$}, denoted by \(\Ch^{  A} _\bd\).
Although this functor does not appear to admit a right adjoint, Berest et al. proved the following proposition.
\begin{proposition}\cite[Theorem 2.6]{BKR2013Der}\label{prop: derived VG}
	Let $A \in \dga_S$.
	Let $\bd$ be a dimension vector.
	The functor $(-)^\GL _\bd$ admits a total left derived functor
	\begin{equation*}
		\bfL (-)^\GL _\bd : \Ho(\dga_S) \to \Ho(\cdga_\bbK).
	\end{equation*}
	Furthermore, for any $A \in \dga_S$, there is an isomorphism
	$\rmH^\bullet (\bfL (A)^\GL _\bd) \cong \rmH^\bullet (\bfL(A)_\bd)^{\GL_\bd (\bbK)}$.
\end{proposition}

Since derived representation schemes and derived character schemes are affine, shifted Poisson structures are constructed by a standard procedure.
Throughout this work, we only consider derived affine schemes \(\Spec R\) which are derived finitely presentable.
We denote by \(\bbL_{\bbK\bcs R}\) the cotangent complex and by \(\bbT_{\bbK\bcs R}\) the tangent complex.
See \cite{Toe2014Der, Saf2017Lec} for details and references.

The space of shifted Poisson structures is defined as follows.
Let \(R\in\cdga_\bbK\) be derived finitely presentable.
Let \(n \in \bbZ\).
The {\it complex of \(n\)-shifted polyvector fields on \(\Spec R\)}
is defined to be
\[
\widehat{\Pol}(R; n)  =  \prod_{k \geq 0} \Sym^k _R ( \bbT_{\bbK \bcs R} [-n-1]).
\]
\begin{definition}[Calaque-Pantev-To\"en-Vaqui\'e-Vezzosi] \label{def: CPTVV's pois}
	Let \(R\in\cdga_\bbK\) be derived finitely presentable.
	Let \(n \in \bbZ\).
	The {\it space of \(n\)-shifted Poisson structures on \(\Spec R\)}
	is 
	\[
	\pois(R; n) : = \underline{\MC}_\bullet (F^2 \widehat{\Pol} ({R}; n)[n+1]). 
	\]
\end{definition}

\subsection{Explicit presentation}\label{subsec: present}

For simplicity,  in this subsection only, we assume \(S = \bbK\).
Let \(A \in \dga_S\). Fix  \(N \in \bbN\) and $V = \bbK^N$.
Denote \((A \ast_\bbK \gl_N(\bbK))^{\gl_N (\bbK)}\) by \(\sqrt[N]{A}\).
Here \((-)^{\gl_N (\bbK)}\) denotes the centralizer of \(\gl_N (\bbK)\).
Then the derived representation scheme \(\DRep^{A} _N\) is represented by
\begin{equation}
	(A)_N : =  (\sqrt[N]{A})_\ab : = \sqrt[N]{A}/([\sqrt[N]{A},\sqrt[N]{A}]),
\end{equation}
\((-)_\ab\) denotes the quotient by the ideal generated by commutators.

Let \({\e_{ij}}\) be the basis of elementary matrices in \(\gl_N (\bbK)\).
For any \(a \in A\), define the ``matrix coefficients'' in \(A \ast \gl_N (\bbK)\):
\[
a_{ij} : = \sum_{k = 1} ^N \e_{ki} a \e_{jk}.
\]
Then \(a_{ij} \in \sqrt[N]{A}\) for all \(i,j=1,\dots,N\).
We also write \(a_{ij}\) for the corresponding element in \(A_N\).
A standard fact is that \(A_N\) is generated by \(\{a_{ij}\mid a\in A,\ i,j=1,\dots,N\}\).

Equivalently, \(a_{ij}\) is a function on the representation scheme:
\[
a_{ij}: \rho \in \DRep^{  A} _N \mapsto (\rho_a)_{ij}.
\] 
Here \(\rho_a\) denotes the evaluation of \(\rho\) on \(a\), and \((\rho_a)_{ij}\) denotes its \((i,j)\)-entry. 
In particular, we set \(\Tr a = \sum_i a_{ii}\).

\begin{example}
	Let \( A = \bbK[x,y,z] \) be the polynomial algebra over \(\bbK\).
	\(A\) admits a cofibrant resolution \( Q = \bbK \langle x,y,z; \xi, \theta, \lambda; t\rangle \), where the generators \(x,y,z\) lie in degree \(0\);
	\(\xi,\theta,\lambda\) lie in degree \(-1\) and \(t\) lies in degree \(-2\). The differential on \(Q\) is defined by
	$$
	\pa \xi = [y,z],\ \pa \theta = [z,x],\ \pa \lambda = [x,y],\ \pa t = [x, \xi] + [y, \theta] + [z, \lambda].
	$$
	For \( V = \bbK^N \), 
	\(Q_N \cong \bbK[x_{ij},\,y_{ij},\,z_{ij};\, \xi_{ij},\, \theta_{ij},\, \lambda_{ij};\,t_{ij}]_{i,j = 1,\dots,N}\ ,\)
	where the generators \(x_{ij}, y_{ij}, z_{ij}\) lie in degree \(0\),
	\(\xi_{ij}, \theta_{ij}, \lambda_{ij}\) lie in degree \(-1\),
	and \(t_{ij}\) lies in degree \(-2\).
	Consider the matrix notations:
	\[
	X =(x_{ij}),\ Y = (y_{ij}),\ Z=(z_{ij}),\ \Xi=(\xi_{ij}),\ \Theta = (\theta_{ij}),\ \Lambda = (\lambda_{ij}),\ T = (t_{ij}).
	\]
	Then the differential on \(Q_N\) is expressed as follows:
	\begin{equation*}
		\pa \Xi = [Y,Z],\ \pa \Theta = [Z,X],\ \pa \Lambda = [X,Y],\ 
		\pa T = [X, \Xi] + [Y, \Theta] + [Z, \Lambda].
	\end{equation*}
\end{example}

Under this explicit presentation,
double Poisson brackets and noncommutative Hamiltonian structures induce classical counterparts on the representation schemes.
\begin{proposition}\label{prop: noncom ham to ham}
	Let \((A, \pa, \ldb-,-\rdb,\bw)\) be a noncommutative Hamiltonian algebra. Let \(N \in \bbN\).
	Then the Poisson bracket on $A_N$ is given by
	\begin{equation}\label{for: double poisson to poisson}
		\{a_{ij},b_{uv}\}
		=\ldb a,b\rdb'_{u j}\,\ldb a,b\rdb''_{i v},
	\end{equation}
	for \(a,b\in A\).
	Furthermore, there is a moment map on \(A_N\)
	\[
	\mu:  \gl_N (\bbK) \to A_N,\qquad \e_{ij} \tos \bw_{ji}.
	\]
\end{proposition}
\begin{proof}
	A proof can be found in \cite[Proposition 7.11.1]{Van2008Double} and \cite[Theorem 6.4.3]{CBEG2007}.
\end{proof}

Note that Proposition \ref{prop: derived VG} implies \(\rmH^0 ( \bfL (A)_N^{\GL}) \cong (A)_N^{\GL}\). The following recovers results from \cite{Cra2011,Van2008Double}.
\begin{proposition}[Berest-Chen-Eshmatov-Ramadoss]\label{prop: CB H_0}
	Let \(A\) be an ordinary NC Poisson algebra.
	For any $N \in \bbN$, there is a unique graded Poisson algebra structure on \(\rmH^\bullet( \bfL (A)_N ^{\GL})\) such that
	\[
	\{\Tr_N \bar a, \Tr_N \bar b\} : =\Tr_N \{ \bar a, \bar b\},\ \text{for any $a,\ b \in \rHC_\bullet(A).$} 
	\]
\end{proposition}

\subsection{Higher traces}\label{subsec: trace}

The trace map establishes a correspondence from noncommutative algebras to functions on their representation schemes.
This correspondence extends to modules over algebras, connecting noncommutative geometry with geometry on representation spaces.

First, we recall the Van den Bergh functor as introduced in \cite{Van2008Quasi, BKR2013Der}. Van den Bergh proposed this functor as a means to connect noncommutative geometry with sheaf theory on representation schemes. Subsequently, Berest et al. extended this construction into the derived setting.

For a dg algebra \(A \in \dga_S\) and dimension vector \(\bd\),
the Van den Bergh functor is
\begin{equation}\label{for: vdb func}
(-)_\bd :\ \Mod(A^{e}) \to \Mod(A_\bd), \qquad M \tos  M_\bd \cong \big(\gl(\bbK^\bd) \otimes A_\bd\big) \otimes_{A^{e}} M.
\end{equation}

The following lemma is the module-theoretic analogue of Proposition \ref{prop: alg duality}.
\begin{lemma}
	Let \(A \in \dga_S\).
	Let \(\bd\) be a dimension vector.
	For any $M \in \Mod(A^{e})$, and
	$L \in \Mod(A_\bd)$, there is a canonical isomorphism of complexes
	\begin{equation}\label{for: mod duality}
		\dgHom_{A_\bd} \big(M_\bd,\,L\big)
		 \cong \dgHom_{A^{e}} \big(M,\;\gl(\bbK^\bd)\otimes L\big).
	\end{equation}
\end{lemma}

Next, we recall the trace map for bimodules:
\begin{equation}\label{for: trace formula}
	\Tr_\bd: M \to \gl_\bd (\bbK) \otimes M_\bd \to M_\bd.
\end{equation}
The first arrow is the canonical map corresponding to the identity under the isomorphism (\ref{for: mod duality}) for \(L = M_\bd\);
the second arrow sends a matrix
to its trace.

According to \cite[Lemma 3.2.2]{Van2008Quasi}, for any \(M \in \Mod(A^e)\), 
\((T_A M)_\bd \simeq \Sym_{A_\bd} (M)_\bd\).
Thus, for any \(n\), taking \(M=\dgDDer_S(A)\) yields
\((\Ts{n}_{  A})_\bd \simeq \Pol(A_\bd; n-1)\).
We now show that shifted double Poisson structures satisfy the Kontsevich--Rosenberg principle.

\begin{proposition}\label{prop: nc pois to pois}
	Let \(A \in \dga_S\) be homologically smooth.
	Let \(n \in \bbZ\).
	For any dimension vector \(\bd\),
	there is a canonical morphism  
	\[
	\dpois(A; n)\to \pois\Bigl( A_\bd ; n\Bigr).
	\]
\end{proposition}

\begin{proof}
	Let \(Q \cof A\) be a cofibrant resolution.
	We show that there is a morphism of dg Lie algebras
	\[
	\TTs{n+1}_{  Q,\,cyc}[n+1] \to \widehat{\Pol}\Bigl(Q_\bd; n\Bigr)[n+1].
	\]
	According to \cite[Proposition 7.6.1]{Van2008Double}, the double Schouten--Nijenhuis bracket on the shifted noncommutative cotangent bundle and the Schouten--Nijenhuis bracket on shifted polyvector fields on \(\DRep_\bd^{  Q}\) are related by the following formula: for any \(X,Y\in\TTs{n+1}_{  Q,\,cyc}[n+1]\),
	\begin{equation}
		\{X_{i j}, Y_{u v}\} =  \ldb X, Y\rdb^{'} _{u j} \ldb X, Y\rdb^{''} _{i v}.
	\end{equation}
	Using (\ref{for: trace formula}), we have
	\begin{align*}
		\{\Tr_\bd (X), \Tr_\bd(Y)\} & = \sum_{i,j} \{X_{i i}, Y_{j j}\}\\
		&=\sum_{i,j} \ldb X, Y\rdb^{'} _{j i} \ldb X, Y\rdb^{''} _{i j}\\
		&= \sum_{i}\{X, Y\}_{i i}\\
		&= \Tr_\bd(\{X, Y\}).
	\end{align*}
	
	Thus, the trace maps preserve dg Lie algebra structures.
	Therefore, at the level of moduli, we obtain a natural simplicial map
	\[
	{\Tr}_\bd: \dpois(Q; n)\to \pois\Bigl(Q_\bd ; n\Bigr).
	\]
\end{proof}

We end this subsection by presenting an explicit correspondence in the quiver setting.
\begin{example}\label{eg: nc poi to pois}
	Consider the Jordan quiver $Q$. Its doubled quiver $\DQ$  is as follows:
	
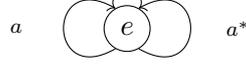
\begin{equation*}
		\begin{tikzpicture}[x={(2cm,0cm)}, y={(0cm,2cm)}, baseline=0cm]
			\node[draw, circle, fill=white] (V) at (0,0) {$e$};
			\draw[->, looseness=5] (V.240) to[out=210,in=150] (V.120);
			\node at (-0.73,0) {\scriptsize $a$};
			\draw[->, looseness=5] (V.300) to[out=330,in=30] (V.60);
			\node at (0.73,0) {\scriptsize $a^\ast$};
		\end{tikzpicture}
	\end{equation*}
	The double Poisson bivector field is
	\[
	P_{\DQ} = - \si\D_{a} \otimes \si \D_{a^\ast} \in \Ts{1}_{\QQ}.
	\]
	Let \(N\) be a positive integer. Since \(\QQ\) is cofibrant, there is no need to choose a further cofibrant resolution. Then
	\[
	\Tr_N \big( - \si\D_{a} \otimes \si \D_{a^\ast} \big) = \sum_{k,l = 1} ^{N} -\si (\D_a)_{kl} \wedge \si (\D_{a^\ast})_{lk}.
	\]
	$\Tr_N(P_{\DQ})$ is the canonical Poisson bivector field on $\gl_N \times \gl_N$:
	\begin{align*}
		\Tr_N (P_{\DQ}) (a^{\ast } _{ij}, a_{uv}) & =\sum_{k,l = 1} ^{N} - i_{\si (\D_a)_{kl}} i_{\si (\D_{a^\ast})_{lk}} (d a^{\ast}_{ij} da_{uv})\\
		&= - \sum_{k,l = 1} ^{N}\delta_{i,k} \delta_{l, j}\delta_{u, l} \delta_{k, v}\\
		&= - \delta_{u,j} \delta_{i,v}.
	\end{align*}
\end{example}

\subsection{Relation with stacks of representations}\label{subsec: stack}

In Yeung's work \cite{Yeu2022Pre, Yeu2022Shi}, a pre-Calabi-Yau structure on a dg algebra induces a shifted Poisson structure on derived stacks of representations; on the other hand, pre-Calabi-Yau structures are known to be closely related to double Poisson structures.
It is natural to ask whether double Poisson structures directly induce shifted Poisson structures on derived stacks of representations.

First, we recall shifted Poisson structures on derived  stacks of representations.
Since a derived moduli stack of representations is a derived quotient stack, To\"en proposed the following model for its cotangent complex (see \cite{Toe2014Der} for details and references).
Yeung studied the shifted symplectic/Poisson structures in \cite{Yeu2022Shi}.
Let \(\rmG\) be a smooth reductive group scheme acting on a smooth derived affine scheme \(Y=\Spec R\), and let \(X:=[Y/\rmG]\) be the derived quotient stack.
The pullback to \(Y\) of the cotangent complex \(\mathbb L_{X}\) is the homotopy fiber \(\bbL\) of the natural morphism
\[
\rho: \bbL_{Y} \to \mathscr{O}_Y \otimes \mathfrak{g}^{\vee},
\]
which is dual to the infinitesimal action of \(\rmG\) on \(Y\). Since \(\rmG\) acts on \(Y\), it also acts on \(\bbL\).
By definition, in \(\Ho(\Mod(R))\),
\(\bbL\) is equivalent to the complex
\begin{equation*}
	\mathrm{cone}\Big(\mathbb{L}_{Y} \stackrel{\rho}{\to} \mathscr{O}_Y \otimes \mathfrak{g}^{\ast}\Big)[-1].
\end{equation*}

Assume \(R\) is almost cofibrant, meaning that the \(\ka\) differential \(\KA_{\rm com}(R)\in\Mod(R)\) is cofibrant and \(\KA_{\rm com}(R)\to\bbL_{\bbK \bcs R}\) is a quasi-isomorphism. In \cite{Yeu2022Shi}, the model for shifted polyvector fields on \(X=[Y/\rmG]\) is:
\begin{align*}\label{for: pol on drep}
	\Pol^r_{\rm Car}(X; n) := & \bigoplus_{p+q = r} \Bigl(\dgHom_{R} \big(\Sym_R (\KA_{\rm com}(R)[n+1]), R\big)\otimes {\rm CoSym}^q\bigl(\fkg[-n]\bigr)\Bigr)^{\rmG}\\
	= & \bigoplus_{p+q = r}  \Big( \Pol(R ;  n) \otimes {\rm CoSym}^q\bigl(\fkg[-n]\bigr) \Big)^{\rmG}.
\end{align*}
The Lie bracket on 
\(\Pol_{\rm Car}(X; n)\)
is defined by
\begin{equation}\label{for: Lie on polcar}
	\{-, -\} := \{-,-\}_{SN} \otimes \xi,
\end{equation}
where \(\{-,-\}_{SN}\) is the Schouten--Nijenhuis bracket on \(\Pol(R ; n)\) and \(\xi\) is the shuffle product on \({\rm CoSym}(\fkg[-n])\).
Under this model, shifted Poisson structures on derived quotient stacks are given as follows.
Let 
\(\widehat{\Pol}_{\rm Car}(X; n)[n+1] := \prod_{r \geq 0} \Pol^r_{\rm Car}(X; n)[n+1]\)
be the completion with the weight filtration: for any \(k \in \bbZ_{\geq 0}\),
\[
F^k\widehat{\Pol}_{\rm Car}(X; n)[n+1] := \prod_{i \geq 0} \Pol^{k+i}_{\rm Car}(X; n) [n+1].
\]

By the Maurer--Cartan formalism, one has
\begin{definition}
	Let \(R \in \cdga_\bbK\) be derived finitely presented. Let \(\rmG\) be a reductive group scheme acting on \(Y = \Spec R\).
	The \emph{Cartan model of the space of \(n\)-shifted Poisson structures on the derived quotient stack \([Y/\rmG]\)} is 
	\[
	\pois_{\rm Car}(X; n) = \underline{\MC}_\bullet \Bigl(F^2\widehat{\Pol}_{\rm Car}(X; n)[n+1]\Bigr).
	\]
\end{definition}

%

Now, fix a cofibrant resolution \(\pi : Q \cof A\) in \(\dga_S\),
and apply the above analysis to the case of the derived moduli stack of representations
\[
\drep^{ A}_{\bd} := [\DRep^{  Q}_{\bd} / \GL_\bd].
\]
Then \(\bbL = \cone\Bigl( \bbL_{\DRep^{  Q}_{\bd}} \stackrel{\rho}{\to} Q_\bd \otimes \gl^\ast _\bd \Bigr)[-1].\)
A key observation in Yeung's work \cite{Yeu2022Pre} is that the noncommutative cotangent complex of \(A\) should be 
\(\crS_{  A}[-1]\),
where \(\crS_{  A} := \cone(\KA(  A) \to A \otimes A).\)
Applying the derived Van den Bergh functor to \(\crS_{  A}[-1]\), one obtains an isomorphism in \(\Ho(\Mod(Q_\bd))\):
\(\bfL (\, \bfR\pi^\ast\crS_{  A}[-1])_\bd \simeq \bbL.\)

As mentioned at the beginning of this section,
Yeung \cite{Yeu2022Pre} proved that pre-Calabi-Yau structures
induce shifted Poisson structures on derived stacks of representations; in particular, he showed that a shifted double Poisson structure induces a pre-Calabi-Yau structure in \cite[Section 3.2]{Yeu2022Pre}.
The trace compatibility used below is a derived analogue of
Van den Bergh's compatibility between the double Schouten--Nijenhuis
bracket and the Schouten--Nijenhuis bracket on representation schemes.

\begin{proposition}\label{prop: dpois to pois on stack}
	Let $A \in \dga_S$ be homologically smooth.
	Then for any $n \in \bbZ$,
	there is a natural morphism 
	\[
	\dpois(A; n) \rightarrow \pois_{\rm Car}(\drep^{ A}_\bd; n).
	\]
\end{proposition}

\begin{proof}
	Let $Q \cof A$ be a cofibrant resolution.
	We show that there is a morphism in the category \(\DGLA^{\rm gr}\):
	\[
	\TTs{n+1}_{  Q,\,cyc}[n+1] \rightarrow \widehat{\Pol}_{\rm Car}\Bigl(\drep^{A}_\bd; n\Bigr)[n+1].
	\]
	In fact, by the definition of the Lie bracket (\ref{for: Lie on polcar}),
	\({\widehat{\Pol}\bigl(Q_\bd; n\bigr)^{\GL_\bd}[n+1]}\) is a dg Lie subalgebra of
	\(\widehat{\Pol}_{\rm Car}\Bigl(\drep^{ A}_\bd; n\Bigr)[n+1]\); moreover, the image of the trace map
	\[
	\Tr_\bd: \Ts{n+1}_{  Q,\,cyc}[n+1] \to \Pol\Bigl(Q_\bd; n\Bigr)[n+1]
	\]
	lies in the \(\GL_\bd (\bbK)\)-invariant part.
	Therefore, by Proposition \ref{prop: nc pois to pois}, the morphism
	\[
	\dpois(A; n) \rightarrow \pois_{\rm Car}(\drep^{  A}_\bd; n)
	\]
	is induced by the trace map followed by the embedding
	\[
	\widehat{\Pol}\bigl(Q_\bd; n\bigr)^{\GL_\bd }[n+1] \hookrightarrow \widehat{\Pol}_{\rm Car}\Bigl(\drep^{A}_\bd; n\Bigr)[n+1].
	\]
\end{proof}

\section{Noncommutative Poisson Extensions}\label{sec: nc poiss ext}

In this section we systematically examine extensions of noncommutative Poisson structures.
The goal of this section is to develop the mechanism that allows double Poisson and NC Poisson brackets to interact with the cobar--bar construction used in Section \ref{sec: s-bracket}.
We study criteria for extending double Poisson brackets on \(A\) to \(A\langle t\rangle\), for extending NC Poisson structures, and for relating these two extension problems. In particular, we show that noncommutative Poisson extensions are compatible with noncommutative Hamiltonian reduction and fit the Kontsevich--Rosenberg principle.

\subsection{Double Poisson algebra extension}\label{subsec: dpoiss ext}
Let $(A, \ldb-,-\rdb)$ be a dg double Poisson algebra.
We aim to extend the double bracket to the free product \(A\langle t\rangle := A\ast_S S[t]\), where the degree of \(t\) is not assumed to be zero.
By the Leibniz rule, such an extension is determined by the double derivation \(\ldb t, - \rdb\).
We therefore introduce the following definition.

\begin{definition}\label{def: dpoi der}
	Let \((A,\pa,\db)\) be a dg double Poisson algebra.
	A \emph{dg double Poisson derivation on \(A\)}
	is a double derivation \(\Theta\in\dgDDer_S(A)\) such that for any \(a, b \in A\),
	\begin{align*}
		\Theta (\ldb a,b \rdb')\otimes \ldb a,b \rdb'' = & 
		\one^{\kappa_1} \Theta(b)' \otimes \ldb a, \Theta(b)'' \rdb \\
		&  +\one^{\kappa_2} \ldb b, \Theta(a)' \rdb'' \otimes \Theta(a)'' \otimes \ldb b, \Theta(a)' \rdb'; \\
		\pa \cc \Theta = & \one^{|\Theta|} \Theta \cc \pa.
	\end{align*}
\end{definition}
Here, \(\kappa_1,\ \kappa_2\) are determined by the Koszul sign rule:
\begin{align*}
	\kappa_1 & = |\Theta| |\s^n a| + |\Theta(b)'| |\Theta(b)''| + |\Theta(b)'| |\mH_a (\Theta(b)'') |;\\
	\kappa_2 &= |\s^n b| (|\Theta| -n -|a|) + |\mH_b (\Theta(a)' )'| |\mH_b (\Theta(a)')'' \Theta(a)''| +1.
\end{align*}

Let \((A,\pa,\ldb-,-\rdb)\) be a dg double Poisson algebra of degree \(n\).
Assume \(t\) has degree \(n\) and let \(\Theta\) be a dg double derivation of degree zero.
We extend the dg double bracket to the free product \(A\langle t\rangle := A\ast_S S[t]\) by setting
\begin{enumerate}
	\item \(\ldb t,t \rdb := t\otimes 1-1\otimes t\);
	\item \(\ldb t,a \rdb := \Theta(a)\), for \(a\in A\).
\end{enumerate}
The following theorem provides a sufficient condition for this extension to be a dg double Poisson bracket.
\begin{theorem}\label{thm: double ext}
	Let \((A, \pa, \db)\) be a dg double Poisson algebra of degree \(n\).
	If \(\Theta\in\dgDDer_S(A)\) is a dg double Poisson derivation of degree zero on \(A\),
	then \(A\langle t; \Theta \rangle\) is a dg double Poisson algebra of degree \(n\). 
\end{theorem}
\begin{proof}
	For simplicity, we consider only the case \(n=0\).
	It suffices to verify the Jacobi identities: \(\ldb t,t,a\rdb=0\) for every \(a\in A\) and \(\ldb t,t,t\rdb=0\).
	\begin{align*}
		\ldb t,t,a \rdb = & \ldb t,\ldb t,a\rdb \rdb_L + (321)\circ\ldb a, \ldb t,t\rdb \rdb_L +  (123)\circ\ldb t, \ldb a,t\rdb \rdb_L\\
		= & \ldb t, \Theta(a)' \otimes \Theta(a)'' \rdb_L + (321) \cc \ldb a, t \otimes 1 - 1 \otimes t \rdb_L + (123) \cc \ldb t, - \tau \ldb t, a \rdb \rdb_L\\
		= & \ldb t, \Theta(a)' \rdb \otimes \Theta(a)''  - \Theta(a)' \otimes 1 \otimes \Theta(a)'' - \one^{|\Theta(a)'| |\Theta(a)''|}(123) \ldb t, \Theta(a)'' \rdb \otimes \Theta(a)'\\
		= & \ldb t, \Theta(a)' \rdb \otimes \Theta(a)''  - \Theta(a)' \otimes 1 \otimes \Theta(a)'' - \Theta(a)' \otimes \ldb t, \Theta(a)'' \rdb.
	\end{align*}
	Since \(\ldb t,-\rdb\) is a dg double Poisson derivation, applying Definition \ref{def: dpoi der} yields
	\begin{align*}
		\ldb t, \Theta(a)' \rdb \otimes \Theta(a)'' - \Theta(a)' \otimes 1 \otimes \Theta(a)'' - \Theta(a)' \otimes \ldb t, \Theta(a)'' \rdb = 0.
	\end{align*}
	Hence, \(\ldb t,t,a \rdb = 0\). A similar computation yields \(\ldb t,t,t\rdb = 0\).
\end{proof}

Throughout this work, the dg algebra \(A\langle t\rangle\) endowed with the dg double bracket defined above is called the \emph{double Poisson algebra extension of \(A\) by \(\Theta\)}, denoted \(A\langle t;\Theta\rangle\).

\subsection{NC Poisson algebra extension}\label{subsec: nc poi ext}

It is natural to consider the extension problem in the setting of NC Poisson structures.
Moreover, since a double Poisson bracket induces an NC Poisson structure, one must clarify how the two extension procedures interact.
Since an NC Poisson structure is not merely a dg Lie bracket on \(A_\natural\); rather, the adjoint action in each argument is induced by a derivation, we introduce the following definition.
\begin{definition}
	Let \((A,\pa,\{-,-\})\) be an NC Poisson algebra of degree \(n\).
	 An \emph{NC Poisson derivation on \(A\)}  is a derivation \(\nu\in\dgDer_S(A)\) such that 
	 for any \(\bar a, \bar b \in A_\natural\),
	\begin{align*}
		\nu({ \{\bar a, \bar b\} }) =& \{\overline{\nu(a)},\bar b\}+\one^{|\nu| |\s^ n a|}\{\bar a,\overline{\nu(b)}\};\\
		\pa \cc \nu =& \one^{|\nu|} \nu \cc \pa.
	\end{align*}
\end{definition}

Let \((A,\pa,\{-,-\})\) be an NC Poisson algebra of degree \(n\).
Assume \(t\) has degree \(n\), and let \(\nu\) be an NC Poisson derivation of degree zero.
We extend the NC Poisson bracket \(\{-,-\}\) to \(A\langle t\rangle\) by setting, for any \(\bar a\in A_\natural\),
\begin{equation}\label{for: H_0 ext}
	\{\bar t, \bar a \} :=\overline{\, \nu (a) \,},\qquad
	\{\bar t, \bar t\} = 0.
\end{equation}
The following theorem provides a sufficient condition for this extension.
\begin{theorem}\label{thm: h_0 ext}
	Let \((A,\pa,\{-,-\})\) be an NC Poisson algebra of degree \(n\).
	Let \(\nu\in\dgDer_S(A)\) be an NC Poisson derivation of degree zero.
	Then \(A\langle t; \nu \rangle\) is an NC Poisson algebra of degree \(n\).
\end{theorem}
\begin{proof}
	For simplicity, we consider only the case \(n=0\).
	Since anti-symmetry is immediate, we verify the Jacobi identity on \(A\langle t\rangle_\natural\).
	For any \(a,b\in A\), because \(\nu\) is an NC Poisson derivation, \(\{\bar t,\bar a,\bar b\}=0\); and
	\begin{align*}
		\{\bar t, \bar t, \bar a\}
		&= \{\bar t,\{\bar t,\bar a\}\} + \{\bar a,\{\bar t,\bar t\}\} + \{\bar t,\{\bar a,\bar t\}\}\\
		&= \{\bar t,\overline{\nu(a)}\} + 0 - \{\bar t,\overline{\nu(a)}\}\\
		&= 0.
	\end{align*}
	Consequently, \(A\langle t;\nu\rangle\) is an NC Poisson algebra of degree \(n\).
\end{proof}
Throughout this work, the dg algebra \(A\langle t\rangle\) endowed with the bracket defined above is called the \emph{NC Poisson algebra extension of \(A\) by \(\nu\)}, denoted \(A\langle t;\nu\rangle\).

Next, we show that a dg double Poisson algebra extension canonically induces an NC Poisson algebra extension.
\begin{proposition}\label{prop: dpois ext to h0 ext}
	Let \((A,\pa,\ldb-,-\rdb)\) be a dg double Poisson algebra of degree \(n\).
	Then for any dg double Poisson derivation \(\Theta\in\dgDDer_S(A)\),
	\(\theta:=\bm\cc\Theta\) is an NC Poisson derivation.
	Furthermore, if the degree of \(\Theta\) is zero and \(t\) is of degree \(n\),
	then \(A \langle t; \theta \rangle\) is an NC Poisson algebra extension of \(A\).
\end{proposition}
\begin{proof}
	For simplicity, we consider the case \(n = 0\text{ and }|\Theta|=0\).
	It is clear that $\theta = \bm \cc \Theta$ is a derivation on $A$.
	For any \(a,b\in A\), one has
	\begin{align*}
		\theta( \{\bar a, \bar b\} )& = \theta \big(\, \overline{\ldb a, b \rdb' \ldb a, b \rdb''} \,\big)\\
		& = \overline{\theta (\ldb a, b \rdb') \ldb a, b \rdb''}  +   \overline{ \ldb a, b \rdb' \theta(\ldb a, b \rdb'')}.
	\end{align*}
	
	Since \(\Theta\) is a dg double Poisson derivation,
	\begin{align*}
		\theta (\ldb a, b \rdb') \ldb a, b \rdb''  = &  \bm \cc (\bm \otimes \id ) \big(\, \Theta  (\ldb a, b \rdb')\otimes \ldb a, b \rdb'' \,\big)\\
		= & \one^{|\s^n b| |\s^n a| + |\mH_b (\Theta(a)' )'| |\mH_b (\Theta(a)')'' \Theta(a)''| + 1}\ldb b,  \Theta (a)'  \rdb'' \Theta (a)'' \ldb  b,  \Theta (a)' \rdb' \\
		& \quad + \one^{|\Theta(b)'| |\Theta(b)''| + |\Theta(b)'| |\mH_a (\Theta(b)'') |} \Theta(b)' \{a , \Theta (b)''\} .
	\end{align*}
	By anti-symmetry,
	\[
	\ldb a ,b \rdb' \theta (\ldb a , b \rdb'') = - \one^{|\s^n a| |\s^n b| + |\mH_b (a) ''| |\mH_b (a)'|} \ldb b , a \rdb'' \theta(\ldb b, a \rdb')
	\]
	Then in \(A_{\natural}\), one has 
	\begin{align*}
		\theta \{\bar a, \bar b\}  = & \overline{\theta (\ldb a, b \rdb') \ldb a, b \rdb''}  + \overline{ \ldb a, b \rdb' \theta(\ldb a, b \rdb'')}\\
		= & \one^{|\Theta(b)'| |\Theta(b)''| + |\Theta(b)'| |\ldb a, \Theta(b)'' \rdb|}  \overline{\Theta(b)' \{a, \Theta(b)'' \}} \\
		& \quad -\one^{|\s^n a| |\s^n b| + |\mH_b (\Theta(a)')'| |\mH_b (\Theta(a)' )'' \Theta(a)''|}  \overline{\ldb b, \Theta(a)' \rdb''  \Theta(a)'' \ldb b, \Theta(a)' \rdb'} \\
		& \quad + \one^{ |\mH_a (\Theta(b)' )'| |\mH_a (\Theta(b)')'' \Theta(b)''| }  \overline{\ldb a,  \Theta (b)'  \rdb'' \Theta (b)'' \ldb  a,  \Theta (b)' \rdb'}\\
		& \quad - \one^{|\Theta(a)'| |\Theta(a)''| + |\Theta(a)'| |\mH_b (\Theta(a)'') | + |\s^n a| |\s^n b|} \overline{\Theta(a)' \{b , \Theta (a)''\}} \\
		= &  \overline{\{ a,  \Theta (b)' \}\Theta (b)''} +
		\one^{|\Theta(b)'| |\Theta(b)''| + |\Theta(b)'| |\ldb a, \Theta(b)'' \rdb|} \overline{ \Theta(b)' \{a, \Theta(b)'' \}} \\
		& -\one^{|\Theta(a)'| |\Theta(a)''| + |\Theta(a)'| |\mH_b (\Theta(a)'') | + |\s^n a| |\s^n b|} \overline{ \Theta(a)' \{b , \Theta (a)''\} }\\
		&- \one^{|\s^n a| |\s^n b|} \overline{\{b, \Theta(a)' \} \Theta(a)''} \\
		= & \overline{\{a, \theta (b)\} } - \one^{|\s^n a| |\s^n b|} \overline{\{b, \theta(a)\}}\\
		= & \overline{\{\theta(a), b\}} +\overline{ \{a, \theta(b)\}}.
	\end{align*}
	Consequently, \(\theta\) is an NC Poisson derivation.
	The second statement follows from Theorem \ref{thm: h_0 ext}.
\end{proof}

\subsection{Reduction commutes with extension}\label{subsec: red com ext}

Let \(A\) be a noncommutative Hamiltonian algebra and let \(A_\bw\) denote its noncommutative Hamiltonian reduction.
It is natural to ask whether extensions are compatible with reduction.
In general, \(A_\bw\) need not carry a dg double Poisson bracket.
We show that certain dg double Poisson algebra extensions of \(A\) descend to NC Poisson algebra extensions of \(A_\bw\), and that this descent is compatible with the Kontsevich--Rosenberg principle.

Since an NC Poisson extension is determined by the derivation associated with \(t\), the following lemma gives a criterion for when the NC Poisson derivation descends to the noncommutative Hamiltonian reduction.
\begin{lemma}
	Let \((A,\pa,\ldb-,-\rdb,\bw)\) be a noncommutative Hamiltonian algebra of degree \(n\).
	Let \(\Theta\in\dgDDer_S (A)\) be a dg double Poisson derivation.
	If \(\Theta (\bw) \in A \otimes A \bw A + A \bw A \otimes A\),
	then \(\theta  = \bm \cc \Theta\) descends to be an NC Poisson derivation of \(A_\bw\).
\end{lemma}
\begin{proof}
	For simplicity, we consider the case \(n=0\).
	For a dg double Poisson derivation \(\Theta\in\dgDDer_S(A)\), 
	if \(\Theta (\bw)\in A\otimes A\bw A + A\bw A\otimes A\),
	then \(\theta(A\bw A)\subset A\bw A\).
	It is straightforward to verify that \(\theta\in\dgDer_S(A_\bw)\).
	It remains to verify that \(\theta\) is an NC Poisson derivation.
	Write \([a]\in A_\bw\) for the class represented by \(a\in A\).
	Then for any \([a],[b]\in A_\bw\), one has 
	\begin{align*}
		\theta \big( \, {\{ \overline{[a]}, \overline{[b]}\} \, } \big) & = \theta \big( \, \overline{H_{[a]} ([b])}\, \big) =\theta\big( \overline{[H_a (b)]} \big)= \theta\big(\overline{\, [\{a, b\}] \,}\big)\\
		& =\overline{\, [ \theta (\{a, b\})] \,} = \overline{\, [  \{\theta(a), b\} ] \,} + \one^{|\theta| |a|} \overline{\, [  \{a, \theta(b)\} ] \,}\\
		& =\{ \theta (\overline{[a]}), \overline{[b]} \} + \one^{|\theta| |a|}\{\overline{[a]}, \theta (\overline{[b]})\}.
	\end{align*}
	Here the second equality is due to the fact that (\ref{for: nc red for lie}) is a morphism of dg Lie algebras; the fifth equality is due to Proposition \ref{prop: dpois ext to h0 ext}.
\end{proof}

Hence, if \(|\Theta|=0\), then \(A_\bw\langle t;\theta\rangle\) is an NC Poisson algebra extension.
Moreover, their NC Poisson structures are compatible as follows.
\begin{lemma}\label{lem: nc red com wt ext}
	Let \((A,\pa,\ldb-,-\rdb,\bw)\) be a noncommutative Hamiltonian algebra of degree \(n\).
	Let \(\Theta\in\dgDDer_S(A)\) be a dg double Poisson derivation of degree zero.
	Then the canonical projection \(A\langle t;\theta\rangle\to A_\bw\langle t;\theta\rangle\)
	descends to a dg Lie algebra morphism
	\[
	A\langle t;\theta\rangle_\natural  \to A_\bw \langle t;\theta\rangle_\natural.
	\]
\end{lemma}
\begin{proof}
	Note that an NC Poisson structure is a dg Lie bracket \(\{-,-\}\), and the adjoint representation
	$\mathrm{ad}:\,  A_\natural \to \dgEnd_S( A_\natural ) $ factors through \(\varrho: \dgDer_S (A)_\natural \to \dgEnd_S (A_\natural)\).
	Therefore, it suffices to verify that the projection \(pr\) preserves the dg Lie bracket on generators.
	For any \(a,b\in A\), by Proposition \ref{prop: CBEG}, 
	it is clear that \(pr(\{\bar a,\bar b\})=\{pr(\bar a),pr(\bar b)\}=\{\overline{[a]},\overline{[b]}\}\).
	What remains to be checked is \(pr(\{\bar t,\bar a\})=\{[\bar t],pr(\bar a)\}\).
	By Definition \ref{def: dpoi der} and (\ref{for: H_0 ext}),
	\begin{align*}
		\{\bar t, \bar a\} & = \overline{\{t, a\}} = \overline{\theta(a)};\\
		\{ [\bar t], pr (\bar a)\} & = \{ [\bar t], \overline{[a]} \} = \overline{\, [\{t, a\}]\,} = \overline{\, [\theta(a)]\, } = pr\big(\, \{\bar t, \bar a\} \,\big).
	\end{align*}
	Consequently, \(pr\) is a dg Lie algebra morphism.
\end{proof}

In other words,
noncommutative reduction commutes with extension.
Next, we prove that  this commutativity fits into the Kontsevich--Rosenberg principle for ordinary algebras.

\begin{theorem}\label{thm: com cubic}
	Let \((A\in\Alg_S,\ldb-,-\rdb,\bw)\) be a noncommutative Hamiltonian algebra.
	Let \(\Theta\in\dgDDer_S(A)\) be a dg double Poisson derivation of degree zero such that \(\Theta (\bw)\in A\otimes A\bw A + A\bw A\otimes A\).
	Then, for any \(\bd\in\bbN^I\), there is a commutative cube of dg Lie algebras and dg Lie algebra morphisms:
	\begin{equation*}
		\begin{split}
			\xymatrixrowsep{0.8pc}
			\xymatrixcolsep{1.2pc}
			\xymatrix{
				A_\natural\ar[rd]|-{pr} \ar[rr]^-{\rm inclusion}  \ar[dd]|-{\Tr}&& 
				A\langle t; \theta \rangle_\natural  \ar[rd]|-{pr} \ar[dd]|-(0.3){\Tr}\\
				& (A_\bw)_\natural   \ar[dd]|-(0.7){\Tr} \ar[rr]^(0.25){\rm inclusion} 
				&& 
				A_\bw\langle t; \theta \rangle_\natural \ar[dd]|-{\Tr} \\
				A_\bd  ^{\GL_\bd (\bbK)} \ar[rd]|-{pr} \ar[rr]^(0.6){\rm inclusion} && 
				(A\langle t; \theta \rangle)_\bd ^{\GL_\bd (\bbK)}  \ar[rd]|-{pr} \\
				& (A_\bw)_\bd ^{\GL_\bd (\bbK)}    \ar[rr]^-{\rm inclusion} 
				&& (A_\bw \langle t; \theta \rangle)_\bd ^{\GL_\bd (\bbK)}.
			}
		\end{split}
	\end{equation*}
\end{theorem}
\begin{proof}
	The commutativity of 
	\[
	\xymatrix{
		A_\natural \ar[r]  \ar[d] & (A_\bw)_\natural \ar[d]\\
		A_\bd ^{\GL_\bd (\bbK)} \ar[r]          &  (A_\bw )_\bd ^{\GL_{\bd} (\bbK)}
	}
	\]
	follows from Proposition \ref{prop: noncom ham to ham}, Proposition \ref{prop: CB H_0}, and Proposition \ref{prop: CBEG}.
	
	The commutativity of
	\begin{align*}
		\xymatrix{
			A_\natural \ar[r]  \ar[d] & A\langle t ; \theta \rangle_\natural  \ar[d]\\
			(A)_\bd ^{\GL_{\bd} (\bbK)}  \ar[r]         & (A\langle t ; \theta \rangle)_\bd ^{\GL_{\bd} (\bbK)} 
		}
		\quad \text{and} \quad
		& \xymatrix{
			(A_\bw)_\natural  \ar[r]  \ar[d] & A_\bw \langle t ; \theta \rangle_\natural \ar[d]\\
			(A_\bw )_\bd ^{\GL_{\bd} (\bbK)}  \ar[r]         & (A_\bw \langle t ; \theta \rangle)_\bd ^{\GL_{\bd} (\bbK)} 
		}\\
	\end{align*}
	follow from Proposition \ref{prop: noncom ham to ham}, Proposition \ref{prop: CB H_0}, and Proposition \ref{prop: dpois ext to h0 ext}.
	
	The commutativity of 
	\begin{equation*}
		\xymatrix{
			A_\natural \ar[r]  \ar[d] & A\langle t ; \theta \rangle_\natural \ar[d]\\
			(A_\bw)_\natural \ar[r]  & A_\bw \langle t ; \theta \rangle_\natural
		}
	\end{equation*}
	follows from Proposition \ref{prop: CBEG} and Lemma \ref{lem: nc red com wt ext}.
	
	The commutativity of 
	\begin{equation*}
		\xymatrix{
			A\langle t ; \theta \rangle_\natural  \ar[r]  \ar[d] & A_\bw \langle t ; \theta \rangle_\natural  \ar[d]\\
			(A\langle t ; \theta \rangle)_\bd ^{\GL_{\bd}(\bbK)}   \ar[r]          & (A_\bw \langle t ; \theta \rangle)_\bd ^{\GL_{\bd} (\bbK)} 
		}
	\end{equation*}
	follows from Proposition \ref{prop: noncom ham to ham}, Proposition \ref{prop: CB H_0}, Lemma \ref{lem: nc red com wt ext}, Proposition \ref{prop: dpois ext to h0 ext}, and Proposition \ref{prop: CBEG}.
\end{proof}

\section{\(\s\)-constructions}\label{sec: s-bracket}

For a dg double Poisson algebra \(A\), a natural question is  whether the shifted complex \(\s^d A = A[d]\) inherits a dg double Poisson structure.
We give an affirmative answer by introducing the \(\s\)-construction.
As a main application, we explain how double Poisson brackets induce Lie algebra structures on reduced cyclic homology.
This phenomenon may be viewed as a noncommutative analogue of the classical fact that the cohomology of a smooth Poisson manifold carries a graded Lie algebra structure.

\subsection{Main constructions}\label{subsec: s-const}
We begin by recalling the bar and cobar constructions; see \cite{Lod1998Cyc} for further details.
Let $A = S \oplus \bar{A} \in \adga_S$ be an augmented dg $S$-algebra.
On the cocomplete coaugmented graded coalgebra
\(T^c_S(\s\bar{A})=\bigoplus_{n=0}(\s\bar{A})^{\otimes n}
= S\oplus\s\bar{A}\oplus(\s\bar{A})^{\otimes 2}\oplus\cdots,\)
there are two differentials, \(\pa\) and \(b\).
The differential \(\pa\) is induced by the internal differential \(\pa_A\) on \(A\).
More precisely,
$$\pa(\s a_1\otimes\cdots\otimes\s a_n)=\sum_{i=1}^{n} \one ^{|\s a_{<i}|+1}
\s a_1\otimes\cdots\otimes\s(\pa_A a_i)\otimes\cdots\otimes\s a_n.$$
The differential \(b\) is given by:
$$b(\s a_1\otimes\cdots\otimes\s a_n)=\sum_{i=1}^{n-1} \one^{|\s a_{<i+1}| +1}
\s a_1\otimes\cdots\otimes\s(a_i a_{i+1})\otimes\cdots\otimes\s a_n.$$
One checks that \(\pa\cc b+b\cc\pa=0\).
Hence, the total differential \(\pa+b\) is well-defined on \(T_S^c(\s\bar{A})\).
The cocomplete coaugmented dg \(S\)-coalgebra \(T_S^c(\s\bar{A})\) is called the \emph{bar construction of \(A\)}, denoted \(\bfB A\).

Let \(C=S\oplus\bar{C}\) be a coaugmented dg \(S\)-coalgebra.
On the augmented graded \(S\)-algebra
$$T_S(\si\bar{C})=\bigoplus_{n=0}(\si\bar{C})^{\otimes n}
= S\oplus\si\bar{C}\oplus(\si\bar{C})^{\otimes 2}\oplus\cdots,$$
there are two differentials, \(\pa\) and \(\delta\).
\(\pa\) is induced from the internal differential \(\pa_C\) on \(C\) by
$$\pa(\si c_1\otimes\cdots\otimes\si c_n)
:=\sum_{i=1} \one^{|\si c_{< i}| +1}\si c_1\otimes\cdots\otimes\si(\pa_C c_i)\otimes\cdots\otimes\si c_n.$$
The differential \(\delta\) is given by
$$\delta(\si c_1\otimes\cdots\otimes\si c_n)
=\sum_{i=1}^{n} \one^{|\si c_{<i}|+|\si c_{i1}|}
\si c_1\otimes\cdots\otimes\si c_{i1}\otimes\si c_{i2}\otimes\cdots\otimes\si c_n,$$
where \(\Delta(c_i)=c_{i1}\otimes c_{i2}\).
One checks that \(\pa\cc\delta+\delta\cc\pa=0\).
Thus, the total differential \(\pa+\delta\) is well-defined on \(T_S(\si\bar{C})\).
The augmented dg \(S\)-algebra \(T_S(\si\bar{C})\) is called the \emph{cobar construction of \(C\)}, denoted \(\bfOm C\).

A key property of the cobar–bar construction is that it provides canonical cofibrant resolutions in the categories of dg algebras and dg coalgebras.
A proof of the following proposition can be found in \cite[Theorem 2.3.1]{LV2012Alg}.
\begin{proposition}\label{prop: cobar bar resolution}
	Let $A$ be an augmented dg $S$-algebra.
	Let $C$ be a conilpotent dg $S$-coalgebra.
	Then the counit $\epsilon : \bfOm \bfB A \to A$ is a quasi-isomorphism of dg $S$-algebras,
	and the unit $\nu : C \to \bfB \bfOm C$ is a quasi-isomorphism of dg $S$-coalgebras.
\end{proposition}

A key observation is the following:
\begin{lemma}
	Let $A \in \adga_S$.
	Then there is an isomorphism of coaugmented dg coalgebras:
	\(\bfB A \cong S \oplus \s \bar A \langle t \rangle\).
\end{lemma}
\begin{proof}
	As graded spaces, one has \(T_S^c(\s\bar A)\cong S\oplus\s\bar A\langle t\rangle\) via the map
	\[
	\s a_1\otimes\cdots\otimes\s a_n \mapsto \s a_1 t \s a_2 t \cdots t \s a_n.
	\]
	The dg coalgebra structure on \(S\oplus\s\bar A\langle t\rangle\) is transported along this isomorphism.
\end{proof}
As a corollary, the canonical cofibrant resolution \(\bfOm\bfB A\) provided by the cobar–bar construction is isomorphic to
\[
(T_S\,\si(\s\bar A\langle t\rangle),\pa+b+\delta,\otimes)
\]
as dg algebras.
Furthermore, the reason to replace the bar construction by \(S \oplus \s \bar A \langle t \rangle\) is as follows.
Using the shifted multiplication \(\tbm\) (see (\ref{for: tbm})), one can verify that  there exists a free product \(\ast\) on
\(\s \bar A \langle t \rangle\);
this free product is not the same as the underlying tensor product in the bar construction.
For \(\s a,\ \s b \in \s \bar A \langle t \rangle \), one has \(\s a \ast \s b = \one^{d|a| + d}\s (ab)\).

When \(A\) is a dg double Poisson algebra, it is natural to ask whether \(T\si(\s\bar A\langle t\rangle)\) inherits a dg double Poisson structure.
The crucial point is to verify the existence of a double Lie bracket on \(\s\bar A\langle t\rangle\).
It is known that if \(L\) is a dg Lie algebra of degree \(n\), then for any \(d\in\bbZ\), the shifted complex \(\s^d L\) is a dg Lie algebra of degree \(n-d\).
However, for a dg double Poisson algebra \(A\) of degree \(n\), it does not follow in general that \(\s^d A\) is a dg double Poisson algebra of degree \(n-d\).
Intuitively, one should define a double bracket on \(\s^d A\) by shifting the bracket on \(A\).

\begin{proposition}\label{prop: sdA double  Lie}
	Let \((A,\pa,\db_A)\) be a dg double Lie algebra of degree \(n\).
	Then for any \(d\in\bbZ\), \(\s^d A\) is a dg double Lie algebra of degree \(n\).
\end{proposition}
\begin{proof}
	Define a double bracket on \(\s^d A\) by
	\[
	\db_{\s^d A}:=(\s^{d}\otimes\s^{d})\cc\db_A\cc(\s^{-d}\otimes\s^{-d}): \s^d A \otimes \s^d A \to \s^d A \otimes \s^d A.
	\]
	We verify the double Jacobi identity for \(\db_{\s^d A}\); the other axioms are analogous.
	By definition,
	\begin{align*}
		\ldb -,\ldb -,-\rdb_{\s^d A}\rdb_{\s^d A,L}
		&=(\db_{\s^d A}\otimes\id)\cc(\id\otimes\db_{\s^d A})\\
		&=\one^{d}(\s^d)^{\otimes 3}\cc(\db_A\otimes\id)\cc(\id\otimes\db_A)\cc(\s^{-d})^{\otimes 3}.
	\end{align*}
	Therefore,
	\begin{align*}
		\ldb-,-,-\rdb_{\s^d A}
		&=\sum_{k=0}^2\tau^{-k}\cc\ldb-,\ldb-,-\rdb_{\s^d A}\rdb_{\s^d A,L}\cc\tau^k\\
		&=\sum_{k=0}^2\one^{d}\tau^{-k}\cc(\s^d)^{\otimes 3}\cc(\db_A\otimes\id)\cc(\id\otimes\db_A)\cc(\s^{-d})^{\otimes 3}\cc\tau^k\\
		&=\one^{d}(\s^d)^{\otimes 3}\cc\Big(\sum_{k=0}^2\tau^{-k}\cc((\db_A\otimes\id)\cc(\id\otimes\db_A))\cc\tau^k\Big)\cc(\s^{-d})^{\otimes 3}\\
		&=\one^{d}(\s^d)^{\otimes 3}\cc\ldb-,-,-\rdb_A\cc(\s^{-d})^{\otimes 3}.
	\end{align*}
	Since \(\db_A\) satisfies the double Jacobi identity, \(\db_{\s^d A}\) is a double Lie bracket of degree \(n\).
\end{proof}

Although shifting is not an endomorphism in the category of dg algebras, it nevertheless equips \(\s^d A\) with a shifted multiplication
\begin{equation}\label{for: tbm}
	\s^d\cc\bm\cc(\s^{-d}\otimes\s^{-d}): \s^d A \otimes \s^d A \to \s^d A.
\end{equation}
When no ambiguity arises we denote \(\s^d\cc\bm\cc(\s^{-d}\otimes\s^{-d})\) by \(\tbm\).
Note that \(\s^d A\) is not an algebra in the usual sense, since the shifted multiplication \(\tbm\) has degree \(d\) rather than zero.
When \(A\) admits a dg double Poisson bracket, the dg double Lie bracket from Proposition \ref{prop: sdA double  Lie} is compatible with the shifted multiplication:

\begin{lemma}\label{lem: tw jacobi id}
	Let \((A,\pa,\db_A)\) be a dg double Poisson algebra of degree \(n\).
	Then the following identity holds:
	\begin{equation}\label{for: tw lebniz}
		\begin{split}
			\db_{\s^d A}\cc(\id\otimes\tbm)
			=&\one^{nd}(\id\otimes\tbm)\cc(\db_{\s^d A}\otimes\id)\\
			&\quad+\one^{nd}(\tbm\otimes\id)\cc(\id\otimes\db_{\s^d A})\cc(12).
		\end{split}
	\end{equation}
\end{lemma}
\begin{proof}
	By definition,
	\begin{align*}
		\db_{\s^d A}\cc\tbm
		&={\s^d}^{\otimes 2}\cc\db_A\cc{\s^{-d}}^{\otimes 2}\cc\big[\id\otimes(\s^d\cc\bm\cc{\s^{-d}}^{\otimes 2})\big]\\
		&={\s^d}^{\otimes 2}\cc\db_A\cc{\s^{-d}}^{\otimes 2}\cc(\id\otimes\s^d)\cc(\id\otimes\bm)\cc(\id\otimes\s^d\otimes\s^d)\\
		&={\s^d}^{\otimes 2}\cc\db_A\cc(\id\otimes\bm)\cc(\s^{-d}\otimes\s^{-d}\otimes\s^{-d}).
	\end{align*}
	Axiom (2) in Definition \ref{def: dpois bracket} is equivalent to
	\begin{equation*}
		\db_A\cc(\id\otimes\bm)=(\id\otimes\bm)\cc(\db_A\otimes\id)+(\bm\otimes\id)\cc(\id\otimes\db_A)\cc(12).
	\end{equation*}
	Then
	\begin{align*}
		\db_{\s^d A}\cc\tbm
		&={\s^d}^{\otimes 2}\cc(\id\otimes\bm)\cc(\db_A\otimes\id)\cc(\s^{-d})^{\otimes 3}\\
		&\quad+{\s^d}^{\otimes 2}\cc(\bm\otimes\id)\cc(\id\otimes\db_A)\cc(12)\cc(\s^{-d})^{\otimes 3}.
	\end{align*}
	Tracking the shifted maps and the Koszul signs yields
	\begin{align*}
		&{\s^d}^{\otimes 2} \cc (\id \otimes \bm) \cc (\db _A \otimes \id )  \cc (\s^{-d} \otimes \s^{-d} \otimes \s^{-d})\\
		= & \one^d (\s^d \otimes \id )\cc (\id \otimes \s^d)  \cc (\id \otimes \bm) \cc (\id \otimes \s^{-d} \otimes \s^{-d})\\
		& \quad \cc (\id \otimes \s^{d} \otimes \s^{d}) \cc (\db _A \otimes \id) \cc (\s^{-d} \otimes \s^{-d} \otimes \s^{-d}) \\
		= & \one^d (\s^d \otimes \id) \cc (\id \otimes \tbm) \cc (\id \otimes \s^d \otimes \s^d) \cc (\db _A \otimes \id) \cc (\s^{-d} \otimes \s^{-d} \otimes \s^{-d})\\
		= & (\id \otimes \tbm) \cc(\s^d \otimes \id \otimes \id) \cc (\id \otimes \s^d \otimes \s^d) \cc (\db _A \otimes \id) \cc (\s^{-d} \otimes \s^{-d} \otimes \s^{-d})\\
		= &\one^{dn} (\id \otimes \tbm) \cc (\db _{\s^d A} \otimes \id) \cc (\id \otimes \id \otimes \s^d) \cc (\id \otimes \s^{-d} \otimes \s^{-d})\\
		= & \one^{dn} (\id \otimes \tbm) \cc (\db _{\s^d A} \otimes \id) \cc (\id \otimes \id \otimes \s^d).
	\end{align*}
	In the same fashion, one checks
	\[
	{\s^d}^{\otimes 2} \cc (\bm \otimes \id) \cc (\id \otimes \db_A) \cc (12) \cc (\s^{-d} \otimes \s^{-d} \otimes \s^{-d}) = \one^{nd} (\tbm \otimes \id ) \cc (\id \otimes \db _{\s^d A})\cc (12).
	\]
	This proves the lemma.
\end{proof}

Thus, when \(A\) admits a dg double Poisson bracket,
the above lemma shows that \(\s^d A\) still carries a dg double Poisson bracket.
Moreover, Lemma \ref{lem: dpois to dgla} implies that double Poisson brackets induce Lie brackets on the cyclic quotient.
We verify that this phenomenon persists in the $\s$-construction as well.
The shifted multiplication induces a bracket on \(\s^d A\):
\[
\{-,-\}_{\s^d A}:=\one^{d}\s^d\cc\{-,-\}_A\cc(\s^{-d}\otimes\s^{-d}) : \s^d A \otimes \s^d A \to \s^d A.
\]
We define the  commutator of \(\s^d A\) to be \([\s^d A,\s^d A]:=(\tbm-\tbm \cc\tau)(\s^d A\otimes \s^d A)\).
Following the previous convention,
\((\s^d A)_{cyc}:=\s^d A/[\s^d A,\s^d A]\).

\begin{corollary}\label{cor: sdA dgla}
	Let \((A,\pa,\db_A)\) be a dg double Lie algebra of degree \(n\).
	Then \((\s^d A)_{cyc}\) is a dg Lie algebra of degree \(n-d\) with respect to \(\{-,-\}_{\s^d A}\).
\end{corollary}
\begin{proof}
	The proof follows from the same technique as in Proposition \ref{prop: sdA double  Lie} together with identity (\ref{for: jacobi id}).
\end{proof}

We now return to the cobar–bar construction \(T_S\,\si(\s\bar A\langle t\rangle)\).
If \(A\) is a dg double Poisson algebra of degree \(n\),
by the noncommutative Poisson extension discussed in Section \ref{sec: nc poiss ext},
we may choose the zero double derivation to extend \(\db\) onto \(\s\bar A\langle t\rangle\).

\begin{proposition}\label{prop: lie on bar}
	Let \(A\in\adga_S\) be a dg double Poisson algebra of degree \(n\).
	There is a canonical structure of a dg double Lie algebra of degree \(n\) on \((\s\bar A\langle t\rangle,\pa+b)\).
\end{proposition}
\begin{proof}
	The double bracket \(\db\) on \(\s\bar A\langle t\rangle\) is defined as above.
	It suffices to verify the commutativity between \(b\) and \(\db\).
	Note that there is a free product \(\ast\) on \(\s\bar A\langle t\rangle\),
	and, as explained in the discussion of noncommutative Poisson extensions (see Section \ref{sec: nc poiss ext}), \(\db\) satisfies the Leibniz rule with respect to the free product.
	Hence, the required commutativity reduces to verifying the identity
	\[
	b\ldb\s x,\s y_1 t \s y_2\rdb
	=\ldb b\s x,\s y_1 t \s y_2\rdb+\ldb\s x,b(\s y_1 t \s y_2)\rdb\one^{|\s x|+n}
	\]
	for homogeneous elements, which is representative of the general case.
	
	The left-hand side equals
	\begin{align*}
		b\ldb\s x,\s y_1 t \s y_2\rdb
		&=b\big(\mH_{\s x}(\s y_1)\ast t\s y_2+\s y_1 t\ast\mH_{\s x}(\s y_2)\one^{|\mH_{\s x}||\s y_1|}\big)\\
		&=b\big(\mH_{\s x}(\s y_1)\ast t\s y_2\big)+b\big(\s y_1 t\ast\mH_{\s x}(\s y_2)\one^{|\mH_{\s x}||\s y_1|}\big)\\
		&=b\big(\s\mH_x(y_1)'\otimes\s\mH_x(y_1)'' t\s y_2\one^{|\s x|+|\mH_x(y_1)'|}\big)\\
		&\quad+b\big(\s y_1 t\s\mH_x(y_2)'\otimes\s\mH_x(y_2)''\one^{|\mH_{\s x}||\s y_1|+|\s x|+|\mH_x(y_2)'|}\big)\\
		&=\s\mH_x(y_1)'\otimes\s(\mH_x(y_1)''y_2)\one^{|\s x|+|\mH_x(y_1)'|+|\s\mH_x(y_1)'\otimes\s\mH_x(y_1)''|}\\
		&\quad+\s(y_1\mH_x(y_2)')\otimes\s\mH_x(y_2)''\one^{|\s x|+|\mH_x(y_2)'|+|\mH_{\s x}||\s y_1|+|\s y_1|}.
	\end{align*}
	Since \(b\s x=0\), the right-hand side equals
	\begin{align*}
		\ldb\s x,b(\s y_1 t \s y_2)\rdb\one^{|\mH_{\s x}|}
		&=\ldb\s x,\s(y_1 y_2)\rdb\one^{|\mH_{\s x}|+|\s y_1|}\\
		&=\s\mH_x(y_1)'\otimes\s(\mH_x(y_1)'' y_2)\one^{|\mH_x(y_1)'|+|\mH_{\s y_1}|}\\
		&\quad+\s(y_1\mH_x(y_2)')\otimes\s(\mH_x(y_2)'')\one^{|y_1||\mH_x|+|y_1\mH_x(y_2)'|+|\mH_{\s y_1}|}.
	\end{align*}
	Comparing signs shows the two expressions agree, hence the commutativity holds.
\end{proof}

Furthermore, the cobar–bar construction carries a dg double Poisson structure.
\begin{theorem}\label{thm: cob bar has dpoiss}
	Let \((A\in\adga_S,\pa,\db)\) be a dg double Poisson algebra of degree \(n\).
	Then \((T_S\,\si(\s\bar A\langle t\rangle),\pa+b+\delta,\db)\) is a dg double Poisson algebra of degree \(n\).
\end{theorem}
\begin{proof}
	By Proposition \ref{prop: lie on bar},
	\(\s\bar A\langle t\rangle\) is a dg double Lie algebra with respect to the differential \(\pa+b\).
	Therefore, \(T_S\,\si(\s\bar A\langle t\rangle)\) is a dg double Poisson algebra (with the tensor product) for the differential \(\pa+b\).
	It remains to verify that the differential \(\delta\) commutes with \(\db\).
	By the Leibniz rule,
	it suffices to verify the commutativity on \(\si(\s\bar A\langle t\rangle)\otimes\si(\s\bar A\langle t\rangle)\).
	Choose arbitrary elements \(\si X=\si(\s x_1 t\s x_2\cdots t\s x_u)\) and \(\si Y=\si(\s y_1 t\s y_2\cdots t\s y_v)\).
	By Proposition \ref{prop: sdA double  Lie},
	\[
	\ldb\si X,\si Y\rdb=\one^{|\si X|}(\si\otimes\si)\big(\ldb X,Y\rdb\big).
	\]
	Here \(\ldb X,Y\rdb\) is the double Lie bracket on \(\s\bar A\langle t\rangle\) and satisfies the Leibniz rule.
	Thus, it suffices to verify the following representative identity: for \(X=\s x\) and \(Y=\s y\) in \(\s\bar A\),
	\[
	\delta\ldb\si X,\si Y\rdb=\ldb\delta\si X,\si Y\rdb+\ldb\si X,\delta(\si Y)\rdb\one^{|\si X|+n}.
	\]
	The left-hand side is
	\begin{align*}
		\delta\ldb\si X,\si Y\rdb
		&=\delta(\si\mH_X(Y)'\otimes\si\mH_X(Y)'')\one^{|\si X|+|\mH_X(Y)'|}\\
		&=\delta(\si\mH_X(Y)')\otimes\si\mH_X(Y)''\one^{|\si X|+|\mH_X(Y)'|}\\
		&\quad+\si\mH_X(Y)'\otimes\delta(\si\mH_X(Y)'')\one^{|X|}\\
		&=\big[\si\mH_X(Y)'\otimes\si1\otimes\si\mH_X(Y)''\one^{|\mH_X(Y)'|}\\
		&\quad+\si1\otimes\si\mH_X(Y)'\otimes\si\mH_X(Y)''\big]\one^{|\si X|+|\mH_X(Y)'|}\\
		&\quad+\big[\si\mH_X(Y)'\otimes\si\mH_X(Y)''\otimes\si1\one^{|\mH_X(Y)''|}\\
		&\quad+\si\mH_X(Y)'\otimes\si1\otimes\si\mH_X(Y)''\big]\one^{|X|}\\
		&=\si1\otimes\si\mH_X(Y)'\otimes\si\mH_X(Y)''\one^{|\si X|+|\mH_X(Y)'|}\\
		&\quad+\si\mH_X(Y)'\otimes\si\mH_X(Y)''\otimes\si1\one^{|X|+|\mH_X(Y)''|}.
	\end{align*}
	The right-hand side is
	\begin{align*}
		&\ldb\delta\si X,\si Y\rdb+\ldb\si X,\delta(\si Y)\rdb\one^{|\mH_{\si X}|}\\
		&=\ldb\si1\otimes\si X+\si X\otimes\si1\one^{|X|},\si Y\rdb+\ldb\si X,\si1\otimes\si Y\\
		&\quad +\si Y\otimes\si1\one^{|Y|}\rdb\one^{|\mH_{\si X}|}\\
		&=\si1\check{\otimes}\ldb\si X,\si Y\rdb\one^n+\ldb\si X,\si Y\rdb\check{\otimes}\si1\one^{|\si Y|+|X|}\\
		&\quad+\si1\otimes\ldb\si X,\si Y\rdb+\ldb\si X,\si Y\rdb\otimes\si1\one^{|\mH_{\si X}|+|Y|}\\
		&=\si\mH_X(Y)'\otimes\si1\otimes\si\mH_X(Y)''\one^{n+|X|}\\
		&\quad+\si\mH_X(Y)'\otimes\si1\otimes\si\mH_X(Y)''\one^{|\mH_{\si X}|}\\
		&\quad+\si1\otimes\si\mH_X(Y)'\otimes\si\mH_X(Y)''\one^{|\si X|+|\mH_X(Y)'|}\\
		&\quad+\si\mH_X(Y)'\otimes\si\mH_X(Y)''\otimes\si1\one^{|\mH_X(Y)''|}\\
		&=\si1\otimes\si\mH_X(Y)'\otimes\si\mH_X(Y)''\one^{|\si X|+|\mH_X(Y)'|}\\
		&\quad+\si\mH_X(Y)'\otimes\si\mH_X(Y)''\otimes\si1\one^{|X|+|\mH_X(Y)''|},
	\end{align*}
	Note that \(\check{\otimes}\) means that we consider the inner bimodule structure with respect to the tensor product.
	Consequently, the two sides agree. This completes the proof.
\end{proof}

At the end, since multiplications \(\bm\) and \(\tbm\) commute with above differentials, by the same argument as in Lemma \ref{lem: dpois to dgla}, one has the following corollary.
\begin{corollary}\label{cor:  cobar bar has loday}
	Let \((A\in\adga_S,\pa,\db)\) be a dg double Poisson algebra of degree \(n\).
	Then the cobar-bar construction \(\bfOm \bfB A\) is a dg Loday algebra of degree \(n\).
	Furthermore, the commutator quotient \((\bfOm \bfB A )_{cyc}\) is a dg Lie algebra of degree \(n\).
\end{corollary}

\subsection{Application to preprojective algebra}

First, we briefly recall cyclic homology theory. See \cite{Lod1998Cyc} for further details.
Let \(A \in \adga_S\) be a dg algebra.
The Hochschild complex \(\CH(A)_\bullet\) is a model to compute the Hochschild homology of \(A\),
which is the cohomology of  \(A \Lot_{A^e} A\) (under the cohomological grading).
It is obtained by replacing \(A\) with its bar resolution.
It is a mixed complex with respect to the differential \(\pa + b\) on \(\CH(A)_\bullet\) and Connes' \(\sfB\) operator.
The cyclic complex is defined to be
\(\CC(A)_\bullet : = (\CH(A)_\bullet [u^{-1}], \pa + b +u\sfB)\).
Here, \(u\) is of degree \(2\).

On the other hand, the cyclic permutation operator \(\tau\) acts on the Hochschild complex \(\CH(A)\).
The Connes' complex is defined to be \(C^\lambda (A) _\bullet : = (\coker (id - \tau), \pa + b)\), see \cite[Section 2.1.4]{Lod1998Cyc} for details.
Since we work over an algebraically closed field of characteristic zero, cohomologies \(\rmH^\bullet (\CC(A))\) and \(\rmH^\bullet (C^\lambda(A))\) coincide and are called the cyclic homology \(\HC_\bullet(A)\).
See \cite[Theorem 2.1.4, Theorem 2.1.8]{Lod1998Cyc} for details.

Berest--Khachatryan--Ramadoss introduced the functor
\[
\mC:\DGA_S\to\Com_S,\qquad   A\tos\cone(S_{cyc}\to A_{cyc}).
\]
Here the map \(S_{cyc}\to A_{cyc}\) is induced by the unit map of \(A\).
One of the main results in \cite{BKR2013Der} is that the functor \(\mC\) admits a total left derived functor \(\bfL\mC:\Ho(\DGA_S)\to\Ho(\Com_S)\).
Berest et al. proved that \(\bfL \mC ( S \bcs A)\cong\cone(C^\lambda(S) _\bullet\to C^\lambda (A)_\bullet)\).
Here, by \(\ S \bcs A\), we emphasize that \(A\) is a dg algebra over \(S\).
See \cite[Proposition 3.1]{BKR2013Der} for details.
Combining  Berest--Khachatryan--Ramadoss' result with Corollary \ref{cor: cobar bar has loday},
we obtain
\begin{corollary}\label{cor: dpois to lie}
	Let \((A\in\adga_S,\pa,\db)\) be a dg double Poisson algebra of degree \(n\).
	Then the reduced cyclic homology \(\rHC_\bullet(A)\) carries a canonical graded Lie algebra structure of degree \(n\).
\end{corollary}

\begin{example}
	Let \(Q\) be a finite quiver.
	The Connes' cyclic complex of its path algebra \(\QQ\) is
	\begin{equation}
		C^\lambda (\QQ)_n  := \QQ^{\otimes (n+1)}/\operatorname{im}(\id-\tau)\ , \quad b_n:\,C^\lambda(\QQ)_n \to C^\lambda (\QQ) _{n-1}\ ,
	\end{equation}
	where \(b_n\) is induced by the standard Hochschild differential and \(\tau\) denotes the cyclic permutation.

	On the other hand, \(0\) as a dg algebra over \(\QQ\) admits a canonical cofibrant resolution: \(\QQ\langle x \rangle\),
	\(|x| = -1 \) and the differential \(\pa\) is given by
	\[
	\pa (p_1 x p_2 \cdots xp_r) = \sum_k \one^{k} p_1 x p_2 \cdots p_kp_{k+1}x p_{k+2}\cdots p_r.
	\]
	It is clear that
	\[
	\bfL \mC(\QQ \bcs 0) \cong \cone(\QQ_{cyc} \to \QQ\langle x \rangle_{cyc}) \cong {\frac{\QQ\langle x\rangle }{\QQ + [\QQ\langle x\rangle ,\QQ\langle x\rangle ]}}\cong C^{\lambda} (\QQ)[1]
	\]
	Using the distinguished triangle in \cite[Lemma 3.2, Theorem 3.3]{BKR2013Der}, one has 
	\(\bfL \mC (\bbK \bcs \QQ) \simeq \cone(C^\lambda (\bbK) _\bullet \to C^\lambda  (\QQ) _\bullet)\).
\end{example}

For a finite quiver \(Q\), Crawley-Boevey, Etingof, Ginzburg and Van den Bergh showed the quotient space $\Pi Q / [\Pi Q, \Pi Q]$ carries a Lie bracket.
Note that this Lie bracket indeed descends to the zeroth reduced cyclic homology $\rHC_0(\Pi Q)$.
This raises the following natural questions: (1) Does there exist a Lie algebra structure on the full reduced cyclic homology $\rHC_\bullet(\Pi Q)$?
(2) How can such a Lie algebra structure be explained at the cochain level?
Since there is no canonical double Poisson bracket on \(\Pi Q\), one cannot deduce a graded Lie algebra structure on \(\rHC_\bullet(\Pi Q)\) directly from Section \ref{sec: s-bracket}.
However, we still have the following result.

\begin{theorem}\label{thm: preproj}
	Let \(Q\) be a finite quiver.
	There is a canonical graded Lie algebra structure on the reduced cyclic homology \(\rHC_\bullet(\Pi Q)\).
\end{theorem}
\begin{proof}
	The bracket \(\{-,-\}:=\bm\cc\db\) on the path algebra \(\QQ\) descends to a dg Loday bracket of degree zero on \(\Pi Q\) (see Remark \ref{rk: loday}). 
	By Proposition \ref{prop: lie on bar} and Theorem \ref{thm: cob bar has dpoiss}, brackets \(\db\text{ and }\{-,-\}\) extend to the cobar–bar construction \(\bfOm\bfB\QQ \cong T_S\si(\s\overline{\QQ}\langle t\rangle)\); consequently, the induced Loday bracket \(\{-,-\}\) on \(\Pi Q\) extends canonically to \(\bfOm \bfB \Pi Q \cong T_S \si(\s \overline{\Pi Q} \langle t \rangle )\) and is compatible with the differential.
	Anti-symmetry holds for the induced bracket \(\{-,-\}\) on the quotient \(\big(T_S\si(\s\overline{\Pi Q}\langle t\rangle)\big)_\natural\).
\end{proof}
It is straightforward to verify that the above results hold for deformed preprojective algebras, since for any $\br \in S$, the deformed preprojective algebra \(\Pi^\br Q = \QQ/(\bw - \br) \) is also a noncommutative Hamiltonian reduction.

Finally, it is natural to ask whether,
when \(A \in \adga_S\) is endowed with a  shifted double Poisson structure (see Definition \ref{def: shifted dpois}),
there is a homotopy version of  Corollary \ref{cor: dpois to lie}?
In forthcoming work, we present an affirmative answer by proving the following theorem.
\begin{theorem}
	Let \(A \in \adga_S\) and \(n \in \bbZ\).
	 If \(A\) is endowed with an \(n\)-shifted double Poisson structure, then the reduced cyclic homology \(\rHC (A)\) carries a canonical graded Lie algebra structure of degree \(n\).
\end{theorem}

\bibliographystyle{plain}
\bibliography{bib_noncomquantred}

\end{document}